\documentclass[11pt]{article}

\usepackage[utf8]{inputenc}
\usepackage[width=16cm,height=22cm]{geometry}

\usepackage{cite}
\newcommand{\doi}[1]{\href{https://doi.org/#1}{doi:#1}}
\newcommand{\myurl}[1]{\href{#1}{#1}}

\usepackage{lscape}
\usepackage{colortbl}
\usepackage{xcolor}
\definecolor{lightgreen}{rgb}{0.7,1,0.7}
\definecolor{lightblue}{rgb}{0.7,0.7,1}

\usepackage{amsmath,amssymb,mathtools}
\usepackage[colorlinks, allcolors=blue]{hyperref}

\usepackage{comment}

\newcommand{\eqdef}{\coloneqq}
\newcommand{\bC}{\mathbb{C}}

\newcommand{\bN}{\mathbb{N}}
\newcommand{\bNz}{\mathbb{N}_0}
\newcommand{\bQ}{\mathbb{Q}}
\newcommand{\bR}{\mathbb{R}}

\newcommand{\bZ}{\mathbb{Z}}

\newcommand{\cP}{\mathcal{P}}

\newcommand{\al}{\alpha}

\newcommand{\ga}{\gamma}

\newcommand{\ka}{\varkappa}
\newcommand{\la}{\lambda}

\newcommand{\si}{\sigma}

\newcommand{\schur}{\operatorname{s}}
\newcommand{\powersum}{\operatorname{p}}
\renewcommand{\hom}{\operatorname{h}}

\newcommand{\Homgen}{\operatorname{H}}
\newcommand{\elem}{\operatorname{e}}

\newcommand{\Mul}{\operatorname{Mul}}
\newcommand{\Add}{\operatorname{Add}}

\usepackage{amsthm}
\newtheorem{thm}{Theorem}[section]
\numberwithin{equation}{section}
\newtheorem{prop}[thm]{Proposition}
\newtheorem{lem}[thm]{Lemma}
\newtheorem{cor}[thm]{Corollary}
\theoremstyle{definition}
\newtheorem{defn}[thm]{Definition}
\newtheorem{ex}[thm]{Example}

\newtheorem{rem}[thm]{Remark}
\newtheorem{alg}[thm]{Algorithm}

\newcommand{\columnhom}{\operatorname{H}}

\newcommand{\Hom}{\operatorname{H}}
\newcommand{\Elem}{\operatorname{E}}

\usepackage{tikz}

\usepackage{listings}

\definecolor{codegreen}{rgb}{0,0.6,0}
\definecolor{codegray}{rgb}{0.5,0.5,0.5}
\definecolor{codepurple}{rgb}{0.58,0,0.82}
\definecolor{backcolour}{rgb}{0.95,0.95,0.92}

\lstdefinestyle{mystyle}{
    language=Python,
    morekeywords={matrix, vector, parent},
    backgroundcolor=\color{backcolour},
    commentstyle=\color{codegray},
    basicstyle=\ttfamily,
    keywordstyle=\bfseries,
    numberstyle=\tiny\color{codegray},
    stringstyle=\color{codepurple},
    breakatwhitespace=false,
    breaklines=true,
    captionpos=b,
    keepspaces=true,
    numbers=none,
    numbersep=5pt,
    showspaces=false,
    showstringspaces=false,
    showtabs=false,
    tabsize=2
}

\newenvironment{egornote}{\color{blue}Egor.}{}

\title{Bialternant formula for Schur polynomials\\with repeating variables}

\author{Luis Angel Gonz\'{a}lez-Serrano and Egor A. Maximenko}

\begin{document}
\maketitle

\smallskip

\begin{abstract}
We consider polynomials of the form $\operatorname{s}_\lambda(y_1^{[\varkappa_1]},\ldots,y_n^{[\varkappa_n]})$,
where $\lambda$ is an integer partition,
$\operatorname{s}_\lambda$ is the Schur polynomial associated to $\lambda$,
and $y_j^{[\varkappa_j]}$ denotes $y_j$ repeated $\varkappa_j$ times.
We represent 
$\operatorname{s}_\lambda(y_1^{[\varkappa_1]},\ldots,y_n^{[\varkappa_n]})$
as a quotient
whose the denominator is the determinant of the confluent Vandermonde matrix,
and the numerator is the determinant of some generalized confluent Vandermonde matrix.
We give three algebraic proofs of this formula.


\end{abstract}

\bigskip

\section*{Authors' data}
Luis Angel Gonz\'{a}lez-Serrano, \quad
\myurl{https://orcid.org/0000-0001-6330-1781}, \\
e-mail: lags1015@gmail.com.

\medskip\noindent
Egor A. Maximenko, \quad
\myurl{https://orcid.org/0000-0002-1497-4338}, \\
e-mail: emaximenko@ipn.mx, egormaximenko@gmail.com.

\bigskip\noindent
Escuela Superior de F\'{i}sica y Matem\'{a}ticas, Instituto Polit\'{e}cnico Nacional,
Ciudad de M\'{e}xico, Postal Code 07730,
Mexico.

\section*{Funding}

The authors have been partially supported by Proyecto CONACYT ``Ciencia de Frontera''
FORDECYT-PRONACES/61517/2020,
by CONACYT (Mexico) scholarships,
and by IPN-SIP projects (Instituto Polit\'{e}cnico Nacional, Mexico).

\clearpage
\tableofcontents

\bigskip\medskip
\section{Introduction}

Given a list of variables $x=(x_1,\ldots,x_N)$
and an integer partition $\la=(\la_1,\la_2,\ldots,\la_{l(\la)})$ of length $l(\la)\le N$,
we denote by $\schur_\la(x_1,\ldots,x_N)$
or shortly by $\schur_\la(x)$
the Schur polynomial associated to $\la$,
in the variables $x_1,\ldots,x_N$.
Schur polynomials form a basis of the vector space of symmetric polynomials.
They play an important role in algebraic combinatorics,
representation theory, and related areas, see~\cite{Egge2019,Lascoux2003, Macdonald1995,Stanley2001,Tamvakis2012}.
There are many equivalent definitions for Schur functions and Schur polynomials,
including 
Jacobi--Trudy formula 
(which involves complete homogeneous polynomials):
\begin{equation}
\label{eq:Jacobi_Trudi_formula}
\schur_\la(x)
=
\det
\bigl[\hom_{\la_j - j + k} (x)\bigr]_{j,k=1}^{l(\la)},
\end{equation}
N\"{a}gelsbach--Kostka formula
(also known as dual Jacobi--Trudy formula),
the combinatorial definition
(via Young tableaux),
Giambelli's formula,
Young's raising operators,
etc.
In what follows, we say ``partition'' instead of ``integer partition''.
If $\la$ is a partition of length $<N$, then we extend it by zeros to the length $N$.
It is well known that $\det\bigl[\hom_{\la_j - j + k} (x)\bigr]_{j,k=1}^{l(\la)}$
does not change if $\la$ is extended by zeros.
Therefore,
\eqref{eq:Jacobi_Trudi_formula} is equivalent to
\begin{equation}
\label{eq:Jacobi_Trudi_formula_extended}
\schur_\la(x)
=
\det
\bigl[\hom_{\la_j - j + k} (x)\bigr]_{j,k=1}^{N}.
\end{equation}
For the computational purposes,
one of the most useful recipes is the so-called bialternant Cauchy--Jacobi formula:
\begin{equation}
\label{eq:schur_bialternant}
\schur_\la(x)
=\frac{\det \bigl[ x_k^{N + \la_j - j} \bigr]_{j,k = 1}^{N}}%
{\det \bigl[ x_k^{N - j} \bigr]_{j,k = 1}^{N}}.
\end{equation}
Both sides of~\eqref{eq:schur_bialternant} can be considered as polynomials in variables $x_1,\ldots,x_N$.
Moreover, the right-hand side of \eqref{eq:schur_bialternant} makes sense
if $x_1,\ldots,x_N$ are pairwise different numbers.
Nevertheless, \eqref{eq:schur_bialternant}
is not well defined if $x_1,\ldots,x_N$ are numbers
and some of them are equal.
For example,~\eqref{eq:schur_bialternant} cannot be applied directly to compute~$\schur_{(3,1)}(7,7,8)$.

In this paper, we prove a generalization of~\eqref{eq:schur_bialternant} in the case where some of $x_1,\ldots,x_N$ coincide.
We denote by $y=(y_1,\ldots,y_n)$ the list of pairwise different variables and by $\ka=(\ka_1,\ldots,\ka_n)$ the list of multiplicities.
Then, we use the following short notation for the list of variables with repetitions:
\begin{equation}
\label{eq:y_rep_informal_def}
y^{[\ka]}
\eqdef
\Bigl(y_1^{[\ka_1]},
\ldots,
y_n^{[\ka_n]}
\Bigr)
=\Bigl(\,\underbrace{y_1,\ldots,y_1}_{\ka_1\ \text{times}}\,\,\ldots,\,
\underbrace{y_n,\ldots,y_n}_{\ka_n\ \text{times}}\,\Bigr).
\end{equation}
For example,
consider the usual Schur polynomial in three variables,
associated $\la=(1,1)$:
\[
\schur_{(1,1)}(x_1,x_2,x_3)
=x_1 x_2
+x_1 x_3
+x_2 x_3.
\]
Now, let $y=(y_1,y_2)$ and $\ka=(2,1)$.
In other words, we substitute $(x_1,x_2,x_3)$
by $(y_1,y_1,y_2)$:
\[
\schur_{(1,1)}(y^{[\ka]})
=\schur_{(1,1)}(y_1^{[2]},y_2^{[1]})
=\schur_{(1,1)}(y_1,y_1,y_2)
=y_1^2 + 2 y_1 y_2.
\]
Our objective is to provide an efficient formula for $\schur_\la(y^{[\ka]})$, when $\la$ is a partition containing large components.

We use notation
$\bN\eqdef\{1,2,\ldots\}$
and $\bNz\eqdef\{0,1,2,\ldots\}$.
Given $N$ in $\bNz$, we denote by $\cP(N)$
the set of the partitions of length $\le N$.

We denote by $G_\la(y,\ka)$ the ``generalized confluent Vandermonde matrix'' associated to the partition $\la$, the list of variables $y$ and the list of repetitions $\ka$.
$G_\la(y,\ka)$ is a square matrix of order $N=|\ka|$,
where
$N\eqdef\ka_1+\cdots+\ka_n$.
The components of $G_\la(y,\ka)$ involve some powers of $y_1,\ldots,y_n$ and binomial coefficients,
see details in~Definition~\ref{defn:G}.
For $\la=\emptyset$,
$G_\emptyset(y,\ka)$
is just the usual confluent Vandermonde matrix.
Here is the main result of this paper.

\begin{thm}
\label{thm:bialternant_formula_Schur_rep}
Let $n\in\bN$,
$\ka\in\bN^n$,
$N\eqdef|\ka|$,
$\la\in\cP(N)$,
and $y_1,\ldots,y_n$ be some variables or pairwise different numbers.
Then
\begin{equation}
\label{eq:bialternant_formula_Schur_rep}
\schur_\la(y^{[\ka]})
=\frac{\det G_\la(y,\ka)}{\det G_{\emptyset}(y,\ka)}.
\end{equation}
\end{thm}

This formula or its particular cases appear
in~\cite[Section~2]{Trench1985_2},
\cite[Section~2.5]{BG2005},
\cite[Remark~3.3]{MM2017},
but we have not found a detailed justification for the general situation.
One possible way to derive this formula is to apply L'Hospital's rule several times.
In this paper, we prefer to give more algebraic proofs.

If $1\le r<s\le n$ and $y_r^{[\ka_r]}$ interchanges with $y_s^{[\ka_s]}$,
then the sign of $\det G_\la(y^{[\ka]})$ does not always change,
see details in Remark~\ref{rem:alternating}.
So, now the word ``bialternant'' in the title of this paper does not have exactly the same sense as in~\eqref{eq:schur_bialternant}.

Applying Theorem~\ref{thm:bialternant_formula_Schur_rep}
to $\la=(m)$
we obtain the following particular case.

\begin{cor}[bialternant formula for complete homogeneous polynomials with repeated variables]
\label{cor:hom_rep_bialternant}
Let $\ka\in\bN^n$ and $m\in\bNz$. Then
\begin{equation}
\label{eq:hom_rep_bialternant}
\hom_{m}(y^{[\ka]})
= \frac{\det G_{(m)}(y, \ka)}{\det G_{\emptyset}(y, \ka)}.
\end{equation}
\end{cor}

The paper has the following structure.
In Section~\ref{sec:repeating}, we formalize notation~\eqref{eq:y_rep_informal_def}.
In Section~\ref{sec:basic_facts_hom}, we recall some basic facts about complete homogeneous polynomials.
Section~\ref{sec:def_G} introduces the definition of $G_\la(y,\ka)$.

In Section~\ref{sec:proof_from_Jacobi_Trudi}, we give a proof of Theorem~\ref{thm:bialternant_formula_Schur_rep} starting with Jacobi--Trudi formula~\eqref{eq:Jacobi_Trudi_formula_extended} and using a multiplication by a special matrix.

In Sections~\ref{sec:partial_reduction_of_JT_determinant}
and~\ref{sec:elimination_determinants_hom},
we explain how to transform some special determinants formed from complete homogeneous polynomials,
applying certain elementary operations to the columns of the matrices.
Section~\ref{sec:partial_reduction_of_JT_determinant}
is inspired by Lita~da~Silva~\cite{LitaDaSilva2018}.
A second proof of Theorem~\ref{thm:bialternant_formula_Schur_rep}, given in Section~\ref{sec:proof_Jacobi_Trudi_and_elementary_operations}, also begins with~\eqref{eq:Jacobi_Trudi_formula_extended} but then incorporates transformations from Sections~\ref{sec:partial_reduction_of_JT_determinant}
and~\ref{sec:elimination_determinants_hom}.

Section~\ref{sec:proof_from_classic_bialternant}
contains the third proof of Theorem~\ref{thm:bialternant_formula_Schur_rep}
which starts with the classical bialternant formula~\eqref{eq:schur_bialternant} and then involves some elementary operations on columns.

The second and third proofs are essentially based on the formula for the difference of complete homogeneous polynomials, see~Lita~da~Silva~\cite[Lemma~1]{LitaDaSilva2018};
we state it as Proposition~\ref{prop:hom_LitaDaSilva} below.

In Section~\ref{sec:computation_of_G},
we explain an algorithm to construct $G_\la(y,\ka)$
applying exponentiation by squaring and precomputing some intermediate expressions.

Finally, in Section~\ref{sec:connection_with_plethysm} we show that if all components of $\ka$ are equal
($\ka_1=\ldots=\ka_n$),
then
$\schur_\la(y^{[\ka]})$
can be written in terms of the plethysm operation.

This paper is designed as a firm foundation for the following two topics:

\begin{itemize}

\item
connections between symmetric polynomials and interpolation problems with repeating nodes
(Eisinberg, Fedele~\cite{EisinbergFedele2006}, and Rawashdeh~\cite{Rawashdeh2019}
already noticed that the inverse of the usual Vandermonde matrix can be written in terms of elementary polynomials);

\item 
determinants and minors of banded Toeplitz matrices
generated by Laurent polynomials with multiple roots
(some connections between banded Toeplitz matrices and skew Schur polynomials were established in
\cite{TracyWidom2001,A2012,AGMM2019,BumpDiaconis2002,GarciaGarciaTierz2020,MM2017}).
\end{itemize}

We also hope that there are analogs of Theorem~\ref{thm:bialternant_formula_Schur_rep} for symplectic Schur polynomials and orthogonal Schur polynomials.
Their definitions and bialternant formulas (without repeating variables)
can be found in~\cite[Section~24 and Appendix~A]{FultonHarris1991}
and~\cite[Section~5]{AGMM2019}.

\medskip
\section{Notation for repeating variables}
\label{sec:repeating}

In this section, we suppose that $n\in\bN$.
Given $\ka$ in $\bN^n$,
we denote by $|\ka|$ the sum of the elements of $\ka$:
\[
|\ka|\eqdef\sum_{p=1}^n\ka_p.
\]
Given $\ka\in\bN^n$,
let $\si(\ka)$ be the list of the partial sums of $\ka$:
\begin{equation}
\label{eq:partial_sums_def}
\si(\ka)_p \eqdef \sum_{q=1}^p \ka_q\qquad
(p\in\{0,1,\ldots,n\}).
\end{equation}
In particular, $\si(\ka)_0\eqdef0$ for each $\ka$.

\begin{defn}
\label{def:gamma}
Let $\ka \in \bN^n$, $N \eqdef |\ka|$. For each $p$ in $\{1,\ldots,|\ka|\}$, we define $\ga(\ka, p)$ as the pair $(q, r)$ defined by the following rules:
\[
q \eqdef
\max \bigl\{r\in\{1,\ldots,n\}\colon \sigma(\ka)_{r - 1} < p\bigr\},
\qquad 
r \eqdef p - \si(\ka)_{q - 1}. 
\]
\end{defn}

It is easy to see that $\gamma(\ka, p)$ from Definition~\ref{def:gamma}
is the unique pair $(q,r)$ such that
$q\in\{1,\ldots,n\}$,
$r\in\{1,\ldots,\ka_q\}$,
\begin{equation}
\label{eq:between_partial_sums}
\si(\ka)_{q-1} < p \le \si(\ka)_q,
\end{equation}
and
\begin{equation}
\label{eq: partial_sums_residue}
 p = \si(\ka)_{q - 1} + r.
\end{equation}

\begin{ex}
Let $\ka= (2, 3)$. Then 
\[
\ga(\ka, 1) = (1, 1), \quad
\ga(\ka, 2) = (1, 2), \quad
\ga(\ka, 3) = (2, 1), \quad
\ga(\ka, 4) = (2, 2), \quad
\ga(\ka, 5) = (2, 3).
\]
Indeed, for $p = 3$,
\[
\underbrace{2}_{\si(\ka)_1} < \underbrace{3}_{p} = \underbrace{2}_{\si(\ka)_1} +\underbrace{1}_{r} \leq \underbrace{5}_{\si(\ka)_2}.
\]
\hfill $\qedsymbol$
\end{ex}

\begin{defn}
Let $\ka\in\bN^n$ and $y_1,\ldots,y_n$ be some numbers or variables.
We denote by $y^{[\ka]}$ the list of length $|\ka|$
where $y_1$ is repeated $\ka_1$ times,
$y_2$ is repeated $\ka_2$ times, and so on.
Formally, for every $p$ in $\{1,\ldots,|\ka|\}$,
we put $(q,r)\eqdef\ga(\ka, p)$ and
\[
(y^{[\ka]})_p \eqdef y_q.
\]
\end{defn}

\begin{ex}
Let $y=(y_1,y_2,y_3)$ and $\ka=(3,1,4)$.
Then
\[
y^{[\ka]}=(y_1,y_1,y_1,y_2,y_3,y_3,y_3,y_3).
\]
\hfill $\qedsymbol$
\end{ex}

\medskip
\section{Basic facts about complete homogeneous polynomials}
\label{sec:basic_facts_hom}

In this section, we recall some well-known facts,
see
\cite[Section~2.2]{Egge2019},
\cite[Section~1.2]{Macdonald1995},
\cite[Section~7.5]{Stanley2001}.
Given $m$ in $\bNz$, the complete homogeneous polynomial of degree $m$ in $N$ variables $x_1,\ldots,x_N$ is defined by the following combinatorial formula:
\begin{equation}
\label{eq:hom_def}
\hom_m(x_1,\ldots,x_N)
\eqdef
\sum_{k\in\bN_0^N\colon|k|=m}
x_1^{k_1}\cdots x_N^{k_N}.
\end{equation}
This combinatorial definition
easily implies the following recurrent formula:
\begin{equation}
\label{eq:hom_recur}
\hom_{m+1}(x_1,\ldots,x_N,x_{N+1})
=\hom_{m+1}(x_1,\ldots,x_N)
+x_{N+1}\hom_m(x_1,\ldots,x_{N+1}).
\end{equation}
In what follows, we abbreviate the list $(x_1,\ldots,x_N)$ by $x$.
The generating function of the sequence $(\hom_m(x))_{m=0}^\infty$ is defined as
\[
\Homgen(x)(t)\eqdef
\sum_{m=0}^\infty\hom_m(x)t^m.
\]
The recurrent formula~\eqref{eq:hom_recur} easily implies the following recurrent formula for the generating functions:
\begin{equation}
\label{eq:Homgen_recur}
(1-x_{N+1}t)\Homgen(x,x_{N+1})(t)=\Homgen(x)(t),
\end{equation}
which yields an explicit formula for the generating function:
\begin{equation}
\label{eq:Homgen_explicit}
\Hom(x)(t)
=\frac{1}{\prod_{j=1}^N (1-x_j t)}.
\end{equation}

\begin{prop}[Lita~da~Silva, 2018]
\label{prop:hom_LitaDaSilva}
Let $x_1,\ldots,x_N,y,z$ be some numbers or variables,
and $m\in\bNz$.
Then
\begin{equation}
\label{eq:hom_LitaDaSilva}
(y - z)\hom_m(x_1, \ldots, x_N, y, z)
= \hom_{m + 1}(x_1, \ldots,x_N, y)
- \hom_{m + 1}(x_1, \ldots,x_N, z).
\end{equation}
\end{prop}

\begin{proof}
Lita~da~Silva~\cite[Lemma~1]{LitaDaSilva2018}
proved~\eqref{eq:hom_LitaDaSilva} using~\eqref{eq:hom_def}.
Here is another proof.
We apply~\eqref{eq:Homgen_recur} and get the next relation for generating functions:
\begin{equation}
\label{eq:Homgen_LitaDaSilva}
(y-z)\,t\,\Hom(x_1,\ldots,x_N,y,z)(t)
=\Hom(x_1,\ldots,x_N,y)
-\Hom(x_1,\ldots,x_N,z).
\end{equation}
Equating the coefficients of $t^{m+1}$ we obtain~\eqref{eq:hom_LitaDaSilva}.
\end{proof}

Next, we mention several equivalent formulas for
$\hom_m$ of one variable $t$ repeated $r$ times, i.e.,
$\hom_m\Bigl(\,\underbrace{t,\ldots,t}_{r}\,\Bigr)$.
The trivial case $m<0$
will also be useful.

\begin{prop}
\label{prop:hom_one_variable_with_repetitions}
Let $r\in\bN$ and $m\in\bZ$. Then
\begin{equation}
\label{eq:hom_one_variable_with_repetitions}
\hom_m(t^{[r]})
=
\binom{r+m-1}{r - 1}
\hom_m(t)
=
\begin{cases}
\binom{r+m-1}{r - 1} t^m, & m\ge0;
\\[1ex]
0, & m<0.
\end{cases}
\end{equation}
\end{prop}

\begin{proof}
We apply~\eqref{eq:hom_def} to $x=t^{[\ka]}$.
In this case, all summands in~\eqref{eq:hom_def} coincide, and the number of summands is
$\binom{r+m-1}{r-1}$.
\end{proof}

The same polynomial $\hom_m(t^{[r]})$ can be also expressed through the $(r-1)$th derivative of the monomial $\hom_{m+r-1}(t)$.

\begin{cor}
Let $r\in\bN$ and $m\in\bZ$. Then
\begin{equation}
\label{eq:hom_one_variable_rep_via_derivative}
\hom_m(t^{[r]})
=
\frac{1}{(r-1)!} \hom_{m + r-1}^{(r-1)}(t).
\end{equation}
\end{cor}

\medskip
\section{\texorpdfstring{Definition of the matrix $\boldsymbol{G_{\la}(y,\ka)}$}{Definition of the matrix G}}
\label{sec:def_G}

\begin{defn}[generalized confluent Vandermonde matrix associated to an integer partition and some variables with repetitions]
\label{defn:G}
Let $\ka\in\bN^n$,
$N\eqdef|\ka|$,
and $\la\in\cP(N)$.
We denote by $G_\la(y,\ka)$ the $N\times N$ matrix with the following components.
Given $j,p$ in $\{1,\ldots,N\}$,
we put $(q, r) \eqdef \ga_\ka(p)$ and
\begin{equation}
\label{eq:G_def}
G_\la(y, \ka)_{j,p}
\eqdef 
\begin{cases}
\displaystyle
\binom{N+\la_j - j}{r-1}
y_q^{N+\la_j - j - r + 1},
& \text{if}\ N+\la_j - j - r + 1 \ge 0;
\\[1.5ex]
0, & \text{otherwise}.
\end{cases}
\end{equation}
We assume the usual convention that $\la_j=0$ if $j$ exceeds the length of $\la$.
\end{defn}

Using Proposition~\ref{prop:hom_one_variable_with_repetitions} and Corollary~\ref{eq:hom_one_variable_rep_via_derivative}
we get other equivalent forms of~\eqref{eq:G_def}:
\begin{align}
\label{eq:G_def_hom}
G_\la(y, \ka)_{j,p}
&=
\binom{N+\la_j-j}{r-1}
\hom_{N+\la_j- j - r + 1}(y_{q}),
\\[1ex]
\label{eq:G_def_deriv}
G_\la(y, \ka)_{j,p}
&=
\frac{1}{(r-1)!}\,
\hom_{N+\la_j - j}^{(r-1)}(y_{q}),
\\[1ex]
\label{eq:G_def_hom_rep}
G_\la(y, \ka)_{j,p}
&=
\hom_{N+\la_j- j - r +1}(y_{q}^{[r]}).
\end{align}

\begin{rem}
With the zero-based numbering ($0 \leq j,p < N$ and $0\leq r < \ka_q$), \eqref{eq:G_def} takes the following form:
\begin{equation}
\label{eq:G0_def}
G_\la(y, \ka)_{j,p}
\eqdef 
\begin{cases}
\displaystyle
\binom{N+\la_j - j - 1}{r}
y_q^{N+\la_j - j - r - 1},
& \text{if}\ N + \la_j - j - r - 1 \ge 0;
\\[1.5ex]
0, & \text{otherwise}.
\end{cases}
\end{equation}
\end{rem}

\begin{ex}
For $\ka = (3, 2)$, $G_\la(y, \ka)$ is
\[
\begin{bmatrix}
\binom{\la_1 + 4}{0}\hom_{\la_1 + 4}(y_1) & \binom{\la_1 + 4}{1}\hom_{\la_1 + 3}(y_1) & \binom{\la_1 + 4}{2}\hom_{\la_1 + 2}(y_1) & \binom{\la_1 + 4}{0}\hom_{\la_1 + 4}(y_2) & \binom{\la_1 + 4}{1}\hom_{\la_1 + 3}(y_2)
\\[1ex]
\binom{\la_2 + 3}{0}\hom_{\la_2 + 3}(y_1) & \binom{\la_2 + 3}{1}\hom_{\la_2 + 2}(y_1) & \binom{\la_2 + 3}{2}\hom_{\la_2 + 1}(y_1) & \binom{\la_2 + 3}{0}\hom_{\la_2 + 3}(x_4) & \binom{\la_2 + 3}{1}\hom_{\la_2 + 2}(y_2)
\\[1ex]
\binom{\la_3 + 2}{0}\hom_{\la_3 + 2}(y_1) & \binom{\la_3 + 2}{1}\hom_{\la_3 + 1}(y_1) & \binom{\la_3 + 2}{2}\hom_{\la_3}(y_1) & \binom{\la_3 + 2}{0}\hom_{\la_3 + 2}(y_2) & \binom{\la_3 + 2}{1}\hom_{\la_3 + 1}(y_2)
\\[1ex]
\binom{\la_4 + 1}{0}\hom_{\la_4 + 1}(y_1) & \binom{\la_4 + 1}{1}\hom_{\la_4}(y_1) & \binom{\la_4 + 1}{2}\hom_{\la_4 - 1}(y_1) & \binom{\la_4 + 1}{0}\hom_{\la_4 + 1}(y_2) & \binom{\la_4 + 1}{1}\hom_{\la_4}(y_2)
\\[1ex]
\binom{\la_5}{0}\hom_{\la_5}(y_1) & \binom{\la_5}{1}\hom_{\la_5 - 1}(y_1) & \binom{\la_5}{2}\hom_{\la_5 - 2}(y_1) & \binom{\la_5}{0}\hom_{\la_5}(y_2) & \binom{\la_5}{1}\hom_{\la_5 - 1}(y_2)
\end{bmatrix}.
\]
\hfill $\qedsymbol$
\end{ex}

\begin{rem}[confluent Vandermonde matrix]
Consider the particular case where $\la=\emptyset$.
$G_{\emptyset}(y,\ka)$ is called the
\emph{confluent Vandermonde matrix} associated to the points $y_1,\ldots,y_n$ and multiplicities $\ka_1,\ldots,\ka_n$.
It is well known 
(see, for example, \cite{HaGibson1980}) that
\begin{equation}
\label{eq:determinant_vandermonde_confluent}
\det G_{\emptyset}(y,\ka)
= (-1)^{N(N-1)/2}
\prod_{1\le j<k\le n}
(y_k - y_j)^{\ka_j \ka_k}.
\end{equation}
If the rows of this matrix are written in the inverse order,
like in~\cite{HaGibson1980},
then the factor $(-1)^{N(N-1)/2}$ disappears.
Notice that
\[
\frac{N(N-1)}{2}
=
\frac{1}{2}
\left(\sum_{j=1}^n \ka_j\right)
\left(\sum_{k=1}^n \ka_k\right)
-
\frac{1}{2}
\sum_{j=1}^n \ka_j
\\
=\sum_{j=1}^n \frac{\ka_j^2-\ka_j}{2}
+ \sum_{j=1}^n\sum_{k=j+1}^n
\ka_j \ka_k.
\]
Therefore,~\eqref{eq:determinant_vandermonde_confluent} admits the following equivalent form:
\begin{equation}
\label{eq:determinant_vandermonde_confluent_another_form}
\det G_{\emptyset}(y,\ka)
= (-1)^{\sum_{j=1}^n \frac{\ka_j (\ka_j-1)}{2}}
\prod_{1\le j<k\le n}
(y_j - y_k)^{\ka_j \ka_k}.
\end{equation}
\end{rem}

\begin{ex}
Let $n = 2$ and $\ka=(3,2)$.
Then
\begin{align*}
G_\emptyset(y,\ka)
=
\begin{bmatrix}
y_1^4 & 4 y_1^3 & 6 y_1^2 & y_2^4 & 4y_2^{3} \\
y_1^3 & 3y_1^{2} & 3y_1 & y_2^3 & 3y_2^2 \\
y_1^2 & 2y_1 & 1 & y_2^2 & 2y_2 \\
y_1 & 1 & 0 & y_2 & 1 \\
1 & 0 & 0 & 1 & 0
\end{bmatrix},\qquad
\det G_\emptyset(y,\ka)
=(y_2-y_1)^6.
\end{align*}
\hfill$\qedsymbol$
\end{ex}

\begin{rem}
\label{rem:alternating}
The numerator and denominator of the quotient in the right-hand side of~\eqref{eq:schur_bialternant} are alternating polynomials in two senses:
they vanish if $x_r=x_s$ for some $r,s$ with $r\ne s$,
and they change the sign if $x_r$ interchanges with $x_s$.
Now consider the determinants of $G_\la(y^{[\ka]})$ and $G_\emptyset(y^{[\ka]})$
which appear in the numerator and denominator of the quotient in the right-hand side of~\eqref{eq:bialternant_formula_Schur_rep}.
Of course, they vanish if some of $y_1,\ldots,y_n$ coincide.
Suppose that $1\le r<s\le n$ and $y^{[\ka_r]}$ interchanges with $y^{[\ka_s]}$.
It is easy to see that
$G_\la(\ldots,y_s^{[\ka_s]},\ldots,y_r^{[\ka_r]},\ldots)$ can be obtained from
$G_\la(\ldots,y_r^{[\ka_r]},\ldots,y_s^{[\ka_s]},\ldots)$ by applying
$(\ka_r+\ka_s)(\ka_{r+1}+\cdots+\ka_{s-1})+\ka_r \ka_s$ column switching.
Therefore,
\begin{equation}
\label{eq:alternating}
\begin{aligned}
\det G_\la(\ldots,y_s^{[\ka_s]},\ldots,y_r^{[\ka_r]},\ldots)
&=
(-1)^{(\ka_r+\ka_s)(\ka_{r+1}+\cdots+\ka_{s-1})+\ka_r \ka_s}
\\
&\quad\times
\det G_\la(\ldots,y_r^{[\ka_r]},\ldots,y_s^{[\ka_s]},\ldots).
\end{aligned}
\end{equation}
\end{rem}

\medskip
\section{First proof: using Jacobi--Trudi formula and matrix multiplication}
\label{sec:proof_from_Jacobi_Trudi}

This section generalizes the reasoning from Macdonald~\cite[Chapter~1, (3.6)]{Macdonald1995}.

For $1 \leq q \leq n$,
let
$b_q\eqdef [\delta_{q,j}]_{j=1}^n$.
Then the list $(y_1^{[\ka_1]}, \ldots, y_{q-1}^{[\ka_{q-1}]}, y_q^{[\ka_q - r]}, y_{q+1}^{[\ka_{q+1}]}, \ldots, y_n^{[\ka_n]})$ can be written as $y^{[\ka - rb_q]}$.

\begin{prop}\label{prop: sum_hom_by_elem_without_some_variables}
Let $q\in \{1,\ldots, n\}$, $d\in \bZ$. Then
\begin{equation}\label{eq:sum_homrep_by_erep}
\sum_{l\in \bZ } (-1)^l \hom_{d - l}(y^{[\ka]}) \elem_l(y^{[\ka - rb_q]}) = \binom{r + d - 1}{r-1}\; \hom_d(y_q) = \hom_d(y_q^{[r]}).
\end{equation}
The sum in the left-hand side of~\eqref{eq:sum_homrep_by_erep} contains a finite number of nonzero terms. 
\end{prop}

\begin{proof}
For $d < 0$, all expressions in \eqref{prop: sum_hom_by_elem_without_some_variables} are zero. Suppose that $d\geq 0$.
We consider the generating functions $\Hom(x)(t)$  with $y^{[\ka]}$ instead of $x$ and $\Elem(x)(t)$ with $y^{[\ka - rb_q]}$ instead of $x$:
\begin{align}
\label{eq: Homrep}
\Hom(y^{[\ka]})(t)
&=
\sum_{s=0}^\infty\hom_s(y^{[\ka]})t^s 
=
\prod_{m=1}^n \frac{1}{(1 - y_mt)^{\ka_m}}
\\[1ex]
\label{eq: Elemrep}
\Elem(y^{[\ka - rb_q]})(t) 
&= \sum_{l = 0}^{\infty}(-1)^l\elem_l(y^{[\ka - rb_q]})t^l =
(1 + y_qt)^{\ka_q - r} 
\prod_{\substack{1\leq m \leq n \\ m\neq q}} (1 + y_mt)^{\ka_m}.
\end{align}
In fact, the sum over $l$ in~\eqref{eq: Elemrep} can be taken up to $N - r$,
but we prefer to work without this restriction.
Then
\[
    \Hom(y^{[\ka]})(t) \Elem(y^{[\ka - rb_q]})(-t) 
    = 
    \frac{1}{(1 - y_qt)^r}.
\]
In the left-hand side we substitute \eqref{eq: Homrep} and \eqref{eq: Elemrep}, and the fraction in the right-hand side has a well-known expansion into a power series~\cite[Equation 5.56]{GrahamKnuthPatashnik1994}:
\[
\sum_{s = 0}^\infty \hom_s(y^{[\ka]})t^s \sum_{l = 0}^{\infty}(-1)^l\elem_l(y^{[\ka - rb_q]})t^l = \sum_{m = 0}^\infty \binom{r + m - 1}{r-1}\;y_q^{m}t^m.
\]
By picking out the coefficient of $t^d$ on either side, we obtain the result.
\end{proof}

\begin{defn}
Given a list of numbers or variables $x_1,\ldots,x_N$
and $\la$ in $\cP(N)$,
we denote by $J_\la(x)$ the corresponding Jacobi--Trudi matrix:
\begin{equation}
\label{eq:Jacobi_Trudi_matrix}
J_{\la}(x)
\eqdef
\bigl[\hom_{\la_j - j + k} (x)\bigr]_{j,k=1}^N.
\end{equation}
Here, as usual, we put $\la_j\eqdef 0$ for $j>\ell(\la)$.
\end{defn}

\begin{defn}
For every $\ka$ in $\bN^n$,
we denote by $M(y,\ka)$
the $N \times N$ matrix with the following components. Given $k, p$ in $\{1,\ldots,N\}$, we put
$(q,r) \eqdef \ga_{\ka}(p)$ and
\begin{equation}
\label{eq:M_def}
(M(y, \ka))_{k,p}
\eqdef (-1)^{N - k - r + 1} \elem_{N - k - r + 1}(y^{[\ka - rb_q]}).
\end{equation}
\end{defn}

\begin{prop}
\label{prop: product_G_equal_H_M}
Let $\ka\in \bN^n$,
$N\eqdef|\ka|$,
and
$\la\in\cP(N)$.
Then
\begin{equation}
\label{eq:product_G_equal_H_M}
G_{\la}(y, \ka)
= J_{\la}(y^{[\ka]}) M(y, \ka).
\end{equation}
\end{prop}

\begin{proof}
Let $j,p \in \{1,\ldots, N\}$
and $(q,r) = \ga_{\ka}(p)$.
We consider the $(j,p)$th component of the product $J_\la(y^{[\ka]}) M(y,\ka)$:
\begin{align*}
(J_{\la}(y^{[\ka]})
M(y, \ka))_{j,p} 
&=
\sum_{k = 1}^N \hom_{\la_j - j + k} (y^{[\ka]})
(-1)^{N - k - r + 1}
\elem_{N - k - r + 1} (y^{[\ka - rb_q]}).
\end{align*}
The sum can be extended to $\bZ$ because $\elem_{N - k - r + 1} (y^{[\ka - rb_q]})=0$ for $k<0$ or $k>N$:
\begin{align*}
(J_{\la}(y^{[\ka]})
M(y, \ka))_{j,p} 
&=
\sum_{k \in \bZ} \hom_{\la_j - j + k} (y^{[\ka]})
(-1)^{N - k - r + 1}
\elem_{N - k - r + 1} (y^{[\ka - rb_q]}).
\end{align*}
Make the change of variables
$l = N - k -r + 1$:
\begin{align*}
(J_{\la}(y^{[\ka]})
M(y, \ka))_{j,p} 
&=
\sum_{l \in \bZ} \hom_{\la_j - j + N - l - r + 1} (y^{[\ka]}) (-1)^{N - k - r + 1}
\elem_{l} (y^{[\ka - rb_q]}) .
\end{align*}
By Proposition~\ref{prop: sum_hom_by_elem_without_some_variables}, this sum equals
$\hom_{\la_j - j + N - r + 1}(y_q^{[r]})$,
which is
$G_\la(y, \ka)_{j,p}$.
\end{proof}

\begin{proof}[Proof of Theorem \ref{thm:bialternant_formula_Schur_rep}]
Take determinants in
\eqref{eq:product_G_equal_H_M}:
\begin{equation}
\label{eq:detG_detJ_detM}
\det G_{\la}(y, \ka)
= \det J_{\la}(y^{[\ka]})
\det M(y, \ka).
\end{equation}
First, consider
$\la = \emptyset$.
Then $J_{\emptyset}(y^{[\ka]})$ is a unitriangular matrix.
Therefore,
$\det J_{\emptyset}(y^{[\ka]}) = 1$ and 
\begin{equation}
\label{eq:detM}
\det M(y,\ka)
= \det G_{\emptyset}(y^{[\ka]}).
\end{equation}
For general $\la$,
we apply~\eqref{eq:Jacobi_Trudi_formula_extended},
\eqref{eq:detG_detJ_detM},
and~\eqref{eq:detM}:
\[
 \schur_{\la}(y^{[\ka]})
 = \det J_\la(y^{[\ka]})
 =
 \frac{\det G_\la(y,\ka)}%
 {\det M(y,\ka)}
 = \frac{\det G_\la(y,\ka)}%
 {\det G_{\emptyset}(y,\ka)}.
 \qedhere
\] 
\end{proof}

\begin{ex}
\label{ex:11_21}
Let $n=2$,
$\ka=(2,1)$,
and $y=(y_1,y_2)$.
Then
\[
M(y, \ka)
=
\begin{bmatrix}
\elem_2(y_1,y_2) & -\elem_1(y_2) & \elem_2(y_1,y_1) \\
- \elem_1(y_1,y_2) & \elem_0(y_2) & - \elem_1(y_1,y_1) \\
\elem_0(y_1,y_2) & 0 & \elem_0(y_1,y_1) \\
\end{bmatrix}
=
\begin{bmatrix}
y_1y_2 & -y_2 & y_1^2 \\
-(y_1 + y_2) & 1 & -2y_1 \\
1 & 0 & 1 \\
\end{bmatrix}.
\]
According to~\eqref{eq:determinant_vandermonde_confluent},
\[
\det M(y,\ka)
=\det G_{\emptyset}(y,\ka)
=-(y_1 - y_2)^2.
\]
Take $\la=(1,1)$.
We identify this partition with $(1,1,0)$.
Then
\begin{align*}
G_\la(y,\ka)
&=
\begin{bmatrix}
\binom{3}{0} \hom_{3}(y_1) & \binom{3}{1}\hom_{2}(y_1) & \binom{3}{0}\hom_{3}(y_2)
\\[0.5ex]
\binom{2}{0} \hom_{2}(y_1) & \binom{2}{1}\hom_{1}(y_1) & \binom{2}{0}\hom_2(y_2)
\\[0.5ex]
\binom{0}{0}\hom_{0}(y_1) & \binom{0}{1}\hom_{-1}(y_1) & \binom{0}{0}\hom_{0}(y_2)
\end{bmatrix}
=
\begin{bmatrix}
y_1^3 & 3\,y_1^2 & y_2^3 \\
y_1^2 & 2\,y_1 & y_2^2 \\
1 & 0 & 1
\end{bmatrix},
\end{align*}
and
\begin{align*}
J_\la(y^{[\ka]})
&=
\begin{bmatrix}
\hom_1(y_1,y_1,y_2) &
\hom_2(y_1,y_1,y_2) &
\hom_3(y_1,y_1,y_2)
\\
\hom_{0}(y_1,y_1,y_2) &
\hom_1(y_1,y_1,y_2) &
\hom_2(y_1,y_1,y_2) &
\\
\hom_{-2}(y_1,y_1,y_2) &
\hom_{-1}(y_1,y_1,y_2) &
\hom_0(y_1,y_1,y_2)
\end{bmatrix}
\\[1ex]
&=
\begin{bmatrix}
2y_1 + y_2 & 3y_1^2 + 2y_1y_2 + y_2^2 & 4y_1^3 + 3y_1^2y_2 + 2y_1y_2^2 + y_2^3 \\
1 & 2y_1 + y_2 & 3y_1^2 + 2y_1y_2 + y_2^2 \\
0 & 0 & 1
\end{bmatrix}.
\end{align*}
The determinants of these matrices are easy to compute:
\[
\det G_\la(y,\ka)
= -y_1(y_1 - y_2)^2(y_1 + 2y_2),\qquad
\det J_\la(y^{[\ka]})
=
y_1(y_1 + 2y_2).
\]
Therefore,
\[
\schur_{(1,1)}(y_1,y_1,y_2)
=\frac{\det G_\la(y,\ka)}{\det M(y,\ka)}
=y_1^2+2y_1 y_2
=\det J_\la(y^{[\ka]}).
\]
\hfill$\qedsymbol$
\end{ex}

\medskip
\section{Partial reduction of the Jacobi--Trudi determinant}
\label{sec:partial_reduction_of_JT_determinant}

Given a square matrix,
we employ a short notation for the elementary column operations:
\begin{itemize}
\item $\Mul(j, \al)$
means ``multiply the $j$th column by $\al$ and divide the determinant by $\al$''.
\item $\Add(j, k, \al)$
means ``add to the $j$th column the $k$th column multiplied by $\al$''.
\end{itemize}
In this section,
we start with the determinant
of $J_\la(x_1,\ldots,x_N)$,
apply some elementary operations of type $\Add(j, k, \al)$,
and simplify the results with~\eqref{eq:hom_recur}.
Thereby, we reduce the number of variables in some columns of the Jacobi--Trudi determinant.

We use the following short notation for the columns that appear in the Jacobi--Trudi matrix.

\begin{defn}
\label{def:column_h}
Let $N\in\bN$, $q\in\bNz$,
$\la\in\cP(N)$,
$m\in\bN$,
and $t_1,\ldots,t_m$ be some variables or numbers.
We denote by
$\columnhom_{q,\la}^N(t_1,\ldots,t_m)$
the following column vector:
\begin{equation}
\label{eq:column_h}
\columnhom_{q, \la}^N(t_1,\ldots, t_m)
\eqdef \bigl[
\hom_{q+\la_j-j}(t_1,\ldots,t_m) 
\bigl]_{j=1}^{N}.
\end{equation}
\end{defn}

With this notation,~\eqref{eq:hom_recur} 
implies the following vectorial identity:
\begin{equation}
\label{eq:hom_recur_column}
\columnhom_{q+1,\la}^N(x_p,\ldots,x_N)
-x_p
\columnhom_{q,\la}^N(x_p,\ldots,x_N)
=\columnhom_{q+1,\la}^N(x_{p+1},\ldots,x_N).
\end{equation}

\begin{ex}
Let $N = 3$,
$\la$ be a partition of length $\le 3$,
and $x=(x_1,x_2,x_3)$.
Then
\begin{align*}
\schur_\la(x)
&=
\det
\begin{bmatrix}
\hom_{\la_1}(x_1,x_2,x_3) & \hom_{\la_1+1}(x_1,x_2,x_3) &
\hom_{\la_1+2}(x_1,x_2,x_3)
\\
\hom_{\la_2-1}(x_1,x_2,x_3) &
\hom_{\la_2}(x_1,x_2,x_3) &
\hom_{\la_2+1}(x_1,x_2,x_3)
\\
\hom_{\la_3-2}(x_1,x_2,x_3) &
\hom_{\la_3-1}(x_1,x_2,x_3) &
\hom_{\la_3}(x_1,x_2,x_3)
\end{bmatrix}
\\[1ex]
&=
\det \Bigl[
\columnhom_{1,\la}^3(x_1,x_2,x_3)\ 
\columnhom_{2,\la}^3(x_1,x_2,x_3)\ 
\columnhom_{3,\la}^3(x_1,x_2,x_3) 
\Bigr].
\end{align*}
Apply operation $\Add(3,2,-x_1)$ and formula~\eqref{eq:hom_recur_column} with $p=1$ and $q=2$.
Then the determinant transforms to the following form:
\begin{align*}
\schur_\la(x)
&=
\det
\begin{bmatrix}
\hom_{\la_1}(x_1,x_2,x_3) & \hom_{\la_1+1}(x_1,x_2,x_3) &
\hom_{\la_1+2}(x_2,x_3)
\\
\hom_{\la_2-1}(x_1,x_2,x_3) &
\hom_{\la_2}(x_1,x_2,x_3) &
\hom_{\la_2+1}(x_2,x_3)
\\
\hom_{\la_3-2}(x_1,x_2,x_3) &
\hom_{\la_3-1}(x_1,x_2,x_3) &
\hom_{\la_3}(x_2,x_3)
\end{bmatrix}
\\[1ex]
&=
\det
\Bigl[
\columnhom_{1,\la}^3(x_1,x_2,x_3)\ 
\columnhom_{2,\la}^3(x_1,x_2,x_3)\ 
\columnhom_{3,\la}^3(x_2,x_3)
\Bigr].
\end{align*}
Apply operation $\Add(2,1,-x_1)$ and formula~\eqref{eq:hom_recur_column} with $p=1$ and $q=1$.
We obtain
\begin{align*}
\schur_\la(x)
&=
\det
\begin{bmatrix}
\hom_{\la_1}(x_1,x_2,x_3) & \hom_{\la_1+1}(x_2,x_3) &
\hom_{\la_1+2}(x_2,x_3)
\\
\hom_{\la_2-1}(x_1,x_2,x_3) &
\hom_{\la_2}(x_2,x_3) &
\hom_{\la_2+1}(x_2,x_3)
\\
\hom_{\la_3-2}(x_1,x_2,x_3) &
\hom_{\la_3-1}(x_2,x_3) &
\hom_{\la_3}(x_2,x_3)
\end{bmatrix}
\\[1ex]
&=
\det
\Bigl[
\columnhom_{1,\la}^3(x_1,x_2,x_3)\ 
\columnhom_{2,\la}^3(x_2,x_3)\ 
\columnhom_{3,\la}^3(x_2,x_3) 
\Bigr].
\end{align*}
Finally, apply operation
$\Add(3,2,-x_2)$
and formula~\eqref{eq:hom_recur_column}
with $p=2$ and $q=2$:
\begin{align*}
\schur_\la(x)
&=
\det
\begin{bmatrix}
\hom_{\la_1}(x_1,x_2,x_3) & \hom_{\la_1+1}(x_2,x_3) &
\hom_{\la_1+2}(x_3)
\\
\hom_{\la_2-1}(x_1,x_2,x_3) &
\hom_{\la_2}(x_2,x_3) &
\hom_{\la_2+1}(x_3)
\\
\hom_{\la_3-2}(x_1,x_2,x_3) &
\hom_{\la_3-1}(x_2,x_3) &
\hom_{\la_3}(x_3)
\end{bmatrix}
\\[1ex]
&=
\det
\Bigl[
\columnhom_{1,\la}^3(x_1,x_2,x_3)\ 
\columnhom_{2,\la}^3(x_2,x_3)\ 
\columnhom_{3,\la}^3(x_3)\ 
\Bigr].
\end{align*}
\hfill$\qedsymbol$
\end{ex}

\begin{prop}
\label{prop:partial_reduction_JT}
Let $N\in\bN$,
$\la \in \cP(N)$,
and $x=(x_1,\ldots,x_N)$.
Then
\begin{equation}
\label{eq:partial_reduction_JT}
\schur_\la(x)
=
\det
\Bigl[
\hom_{\la_j-j+k}(x_k,\ldots,x_N)
\Bigr]_{j,k=1}^N.
\end{equation}
\end{prop}

\begin{proof}
We start with \eqref{eq:Jacobi_Trudi_formula_extended}: 
\[
\schur_\la(x)
=
\det
\bigl[\hom_{\la_j - j + k} (x)\bigr]_{j,k=1}^N
=
\det \Bigl[\columnhom_{1, \la}^N(x) \ \columnhom_{2, \la}^N(x) \ \dots \ \columnhom_{N - 1, \la}^N(x) \ \columnhom_{N, \la}^N(x) \Bigr].
\]
Then, we apply elementary column operations.
The reasoning has an iterative form
and can be represented as two nested cycles.

\medskip\noindent
For $j$ from $1$ up to $N-1$:\\
\mbox{}\hspace*{4em} For $k$ from $N-1$ down to $j$:\\
\mbox{}\hspace*{8em}
$\Add(k+1,k,-x_j)$
and apply \eqref{eq:hom_recur_column}
with $p=j$ and $q=k$.

\medskip\noindent
It is easy to see that after step $j$ in the exterior cycle, we obtain
\[
\det \Bigl[
\columnhom_{1,\la}^N(x_1,\ldots,x_N)\ 
\ldots\
\columnhom_{j,\la}^N(x_{j},\ldots,x_N)\ 
\columnhom_{j + 1,\la}^N(x_{j+1},\ldots,x_N)\ 
\ldots\
\columnhom_{N,\la}^N(x_{j+1},\ldots,x_N)
\Bigr].
\]
So, after all steps we obtain
\[
\det \Bigl[
\columnhom_{1,\la}^N(x_1,\ldots,x_N)\ 
\columnhom_{2,\la}^N(x_2,\ldots,x_N)\
\ldots\
\columnhom_{N,\la}^N(x_N)
\Bigr],
\]
i.e., the right-hand side of~\eqref{eq:partial_reduction_JT}.
\end{proof}

\medskip
\section{Elimination of variables in certain polynomial determinants}
\label{sec:elimination_determinants_hom}

In this section, we explain an auxiliar idea which will be used in Section~\ref{sec:proof_Jacobi_Trudi_and_elementary_operations}.

Formula~\eqref{eq:hom_LitaDaSilva} implies the following vectorial identity
involving notation from Definition~\ref{def:column_h}:
\begin{equation}
\label{eq:hom_LitaDaSilve_vectorial}
\columnhom_{q+1, \la}^N(t_1,\ldots,t_m,u)
+
(v-u)
\columnhom_{q,\la}^N(t_1,\ldots,t_m,u,v)
=\columnhom_{q+1, \la}^N(t_1,\ldots,t_m,v).
\end{equation}

\begin{ex}
Consider the following determinant:
\begin{align*}
D
&=
\det
\begin{bmatrix}
\hom_{\la_1}(u_1,u_2, v_1,v_2) &
\hom_{\la_1+1}(u_2, v_1,v_2) &
\hom_{\la_1+2}(v_1,v_2) &
\hom_{\la_1+3}(v_2) 
\\
\hom_{\la_2-1}(u_1,u_2, v_1,v_2) &
\hom_{\la_2}(u_2, v_1,v_2) &
\hom_{\la_2+1}(v_1,v_2)
& \hom_{\la_2+2}(v_2)
\\
\hom_{\la_3-2}(u_1,u_2, v_1,v_2) &
\hom_{\la_3-1}(u_2, v_1,v_2) &
\hom_{\la_3}(v_1,v_2) &
\hom_{\la_3+1}(v_2) 
\\
\hom_{\la_4-3}(u_1,u_2, v_1,v_2) &
\hom_{\la_4-2}(u_2, v_1,v_2) &
\hom_{\la_4-1}(v_1,v_2) &
\hom_{\la_4}(v_2) 
\end{bmatrix}
\\[1ex]
&=
\det \Bigl[
\columnhom_{1,\la}^4(u_1,u_2, v_1,v_2)\ 
\columnhom_{2,\la}^4(u_2, v_1,v_2)\ 
\columnhom_{3,\la}^4(v_1,v_2)\ 
\columnhom_{4,\la}^4(v_2)
\Bigr].
\end{align*}
Apply elementary column operations $\Mul(1,u_1-v_1)$ and 
$\Add(1, 2, 1)$. Next, transform the first column as follows:
\[
(u_1-v_1)
\columnhom_{1,\la}^4(u_1,u_2, v_1,v_2)
+
\columnhom_{2,\la}^4(u_2, v_1,v_2)
= \columnhom_{2,\la}^4 (u_1,u_2,v_2).
\]
So,
\[
D =\frac{1}{u_1 - v_1} \det \Bigl[
\columnhom_{2,\la}^4(u_1,u_2, v_2)\ 
\columnhom_{2,\la}^4(u_2, v_1,v_2)\ 
\columnhom_{3,\la}^4(v_1,v_2)\ 
\columnhom_{4,\la}^4(v_2)
\Bigr].
\]
Applying operations $\Mul(2,u_2-v_1)$ and $\Add(2, 3, 1)$, we get 
\[
D = \frac{1}{(u_1 - v_1)(u_2 - v_1)}\det \Bigl[
\columnhom_{2,\la}^4(u_1, u_2, v_2)\ 
\columnhom_{3,\la}^4(u_2, v_2)\ 
\columnhom_{3,\la}^4(v_1, v_2)\ 
\columnhom_{4,\la}^4(v_2)
\Bigr].
\]
Applying operations $\Mul(1,u_1-v_2)$ and $\Add(1, 2, 1)$, we get
\[
D = \frac{\det \Bigl[
\columnhom_{3,\la}^4(u_1, u_2)\ 
\columnhom_{3,\la}^4(u_2, v_2)\ 
\columnhom_{3,\la}^4(v_1, v_2)\ 
\columnhom_{4,\la}^4(v_2)
\Bigr]
}{(u_1 - v_1)(u_2 - v_1)(u_1 - v_2)}.
\]
Finally, applying operations $\Mul(2,u_2-v_2)$ and $\Add(2, 4, 1)$, we get
\[
D =
\frac{\det \Bigl[
\columnhom_{3,\la}^4(u_1, u_2)\ 
\columnhom_{4,\la}^4(u_2)\ 
\columnhom_{3,\la}^4(v_1, v_2)\ 
\columnhom_{4,\la}^4(v_2)
\Bigr]}{(u_1 - v_1)(u_2 - v_1)(u_1 - v_2)(u_2 - v_2)}.
\]
The determinant in the numerator contains only $u_1, u_2$ in the first two columns and $v_1, v_2$ in the last two columns.
\hfill$\qed$
\end{ex}

\begin{lem}
\label{lem:reduction}
Let $N, p, q \in \bN$ such that $p + q \leq N$, $A$ be an $N \times t$ matrix, where $t = N - p - q$, $u = (u_1, \ldots, u_p)$, $v = (v_1,\ldots, v_q)$. Suppose that $u_j \neq v_k$ for all $j$ in $\{1,\ldots, p \}$ and $k$ in $\{1, \ldots, q\}$. Then
    \begin{align*}
    &
    \det 
    \begin{bmatrix}
        A & \columnhom_{t + 1, \la}^N(u_1, \ldots, u_p, v) & \ldots & \columnhom_{t + p, \la}^N(u_p, v) & \columnhom_{t + p + 1,\la}^N(v_1,\ldots,v_q) & \ldots & \columnhom_{t + p + q,\la}^N(v_q)
    \end{bmatrix}\\
    &= \frac{\det 
    \begin{bmatrix}
        A & \columnhom_{t + q + 1, \la}^N(u_1,\ldots,u_p) & \ldots & \columnhom_{t + p + q, \la}^N(u_p) & \columnhom_{t + p + 1,\la}^N(v_1,\ldots,v_q) & \ldots & \columnhom_{t + p + q,\la}^N(v_q)
    \end{bmatrix}}{\prod_{j,k=1}^{p,q}(u_j - v_k)}
    \\
    &= \frac{\det 
    \begin{bmatrix}
        A & \columnhom_{t + p + q, \la}^N(u_p) & \ldots & \columnhom_{t + q + 1, \la}^N(u_1,\ldots,u_p) & \columnhom_{t + p + 1,\la}^N(v_1,\ldots,v_q) & \ldots & \columnhom_{t + p + q,\la}^N(v_q)
    \end{bmatrix}}{(-1)^{\frac{p(p-1)}{2}}\prod_{j,k=1}^{p,q}(u_j - v_k)}
    .   
    \end{align*}
\end{lem}

\begin{proof} We denote by $D$ the determinant on the left-hand side. The idea of the proof is to apply some elementary column operations and formula~\eqref{eq:hom_LitaDaSilve_vectorial}. The proceeding has an iterative form. We will explain the first steps, the general step, and the last step.

Step 1. For $j$ in $\{1, \ldots, p - 1\}$,
    apply operations $\Mul(t + j, u_j - v_1)$ and $\Add(t + j, t + j + 1, 1)$. Next, apply operations $\Mul(t + p, u_p - v_1)$ and $\Add(t + p, t +p + 1, 1)$. Then, we get 
    \[
    D =
    \frac{\det 
    \begin{bmatrix}
        A & \columnhom_{t + 2, \la}^N(u, v_{2:q}) & \ldots & \columnhom_{t + p+1, \la}^N(u_p, v_{2:q}) & \columnhom_{t + p + 1,\la}^N(v) & \ldots & \columnhom_{t + p + q,\la}^N(v_q)
    \end{bmatrix}}{\prod_{j=1}^p(u_j - v_1)}.
    \]
    Step 2. For $j$ in $\{1, \ldots, p - 1\}$
    apply operations $\Mul(t + j, u_j - v_2)$ and $\Add(t + j, t + j + 1, 1)$. Next, apply operations $\Mul(t + p, u_p - v_2)$ and $\Add(t + p, t +p + 2, 1)$. Then, we get 
    
    \[
    D 
    = 
    \frac{
    \det 
    \begin{bmatrix}
        A & \columnhom_{t + 3, \la}^N(u, v_{3:q}) & \ldots & \columnhom_{t + p + 2, \la}^N(u_p, v_{3:q}) & \columnhom_{t + p + 1,\la}^N(v) & \ldots & \columnhom_{t + p + q,\la}^N(v_q)
    \end{bmatrix}}{\prod_{j=1}^p\prod_{k=1}^2(u_j - v_k)}.
    \]
    Step $r$. For $j$ in $\{1, \ldots, p - 1\}$
    apply operations $\Mul(t + j, u_j - v_r)$ and $\Add(t + j, t + j + 1, 1)$. Next, apply operations $\Mul(t + p, u_p - v_r)$ and $\Add(t + p, t +p + r, 1)$. Then, we get 
    \[
    D
    =
    \frac{\det 
    \begin{bmatrix}
        A & \columnhom_{t + r + 1, \la}^N(u, v_{r + 1:q}) & \ldots & \columnhom_{t + p + r, \la}^N(u_p, v_{r + 1:q}) & \columnhom_{t + p + 1,\la}^N(v) & \ldots & \columnhom_{t + p + q,\la}^N(v_q)
    \end{bmatrix}}{
    \prod_{j=1}^p\prod_{k=1}^r(u_j - v_k)}.
    \]
    Step $q$ (last step). For $j$ in $\{1, \ldots, p - 1\}$
    apply operations $\Mul(t + j, u_j - v_q)$ and $\Add(t + j, t + j + 1, 1)$. Next, apply operations $\Mul(t + j, u_p - v_q)$ and $\Add(t + p, t +p + q, 1)$. Then, we get
    \[
    D = \frac{\det 
    \begin{bmatrix}
        A & \columnhom_{t + q + 1, \la}^N(u) & \ldots & \columnhom_{t + p + q, \la}^N(u_p) & \columnhom_{t + p + 1,\la}^N(v) & \ldots & \columnhom_{t + p + q,\la}^N(v_q)
    \end{bmatrix}}{\prod_{j=1}^p\prod_{k=1}^q(u_j - v_k)}.     \] 
Finally, after all these steps, we reverse the order of the columns with indices $t + 1, \ldots, t + p$.
\end{proof}

\medskip
\section{Second proof: using Jacobi--Trudi formula and elementary operations}
\label{sec:proof_Jacobi_Trudi_and_elementary_operations}

This section is essentially based on Sections~\ref{sec:partial_reduction_of_JT_determinant} and~\ref{sec:elimination_determinants_hom}.

\begin{proof}[Second proof of Theorem~\ref{thm:bialternant_formula_Schur_rep}.]
We denote $y^{[\ka]}$ by $x$
and $\schur_\la(x)$ by $D$.
By Proposition~\ref{prop:partial_reduction_JT},
\[
D
=
\det \Bigl[
\columnhom_{1,\la}^N(x_1,\ldots,x_N)\ 
\columnhom_{2,\la}^N(x_2,\ldots,x_N)\
\ldots\
\columnhom_{N,\la}^N(x_N)
\Bigr].
\]
Now, we iteratively define matrices $A_0, A_1, \ldots, A_{n}$. Recall that $\si(\ka)$ denotes the list of the partial sums of $\ka$. For each $q$ in $\{1, \ldots, n\}$, $A_q$ is a matrix of size $N \times \si(\ka)_{q}$. $A_0$ is just the empty matrix,
\begin{align*}
A_1 & \eqdef \begin{bmatrix}
A_0 & \columnhom_{N, \la}^N(x_{\ka_1}) & \ldots & \columnhom_{N - \ka_1 + 1, \la}^N(x_1,\ldots, x_{\ka_1})
\end{bmatrix},
\\
A_2 & \eqdef 
\begin{bmatrix}
A_1 & \columnhom_{N, \la}^N(x_{\ka_1 + \ka_2})
 & \ldots & \columnhom_{N - \ka_2 + 1, \la}^N(x_{\ka_1 + 1},\ldots, x_{\ka_1 + \ka_2})
\end{bmatrix},
\\
A_3 & \eqdef 
\begin{bmatrix}
A_2 & \columnhom_{N, \la}^N(x_{\ka_1 + \ka_2 + \ka_3})
 & \ldots & \columnhom_{N - \ka_3 + 1, \la}^N(x_{\ka_1 + \ka_2 +  1},\ldots, x_{\ka_1 + \ka_2 + \ka_3})
\end{bmatrix}, \ \ldots, 
\\
A_{q} & \eqdef
\begin{bmatrix}
A_{q-1}
&  \columnhom_{N, \la}^N(x_{\si(\ka)_{q}}) & \ldots &\columnhom_{N - \ka_{q} + 1, \la}^N(x_{\si(\ka)_{q-1} + 1},\ldots, x_{\si(\ka)_{q}})
\end{bmatrix}, \ \ldots.
\end{align*}
Moreover, we need the following products of differences:
\[
P_q 
= 
\prod_{j = \ka_1 + \ldots +  \ka_{q-1} + 1}^{\ka_1 + \ldots + \ka_q} 
\; 
\prod_{k = \ka_1 + \ldots + \ka_q + 1}^N (x_j - x_k) 
=
\prod_{j = \si(\ka)_{q-1} + 1}^{\si(\ka)_q } 
\; 
\prod_{k = \si(\ka)_{q} + 1}^N (x_j - x_k).
\]
After that, we will iteratively apply Lemma~\ref{lem:reduction}.

Step 1.
Apply Lemma~\ref{lem:reduction}
with $t=0$ (so, the corresponding block $A_0$ is an empty matrix),
$p=\ka_1$, $ q = N - \ka_1$,
$u=(x_{1},\ldots,x_{\ka_1})$,
$v=(x_{\ka_1 + 1},\ldots,x_{N})$.
We obtain
\begin{align*}
D
&= 
\frac{
\begin{vmatrix}
     \columnhom_{N, \la}^N(x_{\ka_1}) & \ldots &\columnhom_{N - \ka_1 + 1, \la}^N(x_1,\ldots, x_{\ka_1}) & \columnhom_{\ka_1 + 1, \la}^N(x_{\ka_1 + 1}, \ldots, x_{N}) & \ldots & \columnhom_{N, \la}^N(x_{N}) 
\end{vmatrix}}%
{(-1)^{\frac{\ka_1(\ka_1 - 1)}{2}}\,P_1}\\
&= 
\frac{
\begin{vmatrix}
     A_1 & \columnhom_{\ka_1 + 1, \la}^N(x_{\ka_1 + 1}, \ldots, x_{N}) & \ldots & \columnhom_{N, \la}^N(x_{N}) 
\end{vmatrix}}%
{(-1)^{\frac{\ka_1(\ka_1 - 1)}{2}}\,P_1}.
\end{align*}
Step 2.
Apply Lemma~\ref{lem:reduction}
with $t=\ka_1$,
$p=\ka_2$, $q=N - \ka_1 - \ka_2$,
$u=(x_{\ka_1 + 1},\ldots,x_{\ka_1 + \ka_2})$,
$v=(x_{\ka_1 + \ka_2 + 1},\ldots,x_{N})$.
We get
\begin{align*}
D
&=
\frac{
\begin{vmatrix}
    A_1 & \columnhom_{N, \la}^N(x_{\ka_1 + \ka_2}) & \ldots & \columnhom_{N - \ka_2 + 1, \la}^N(
    x_{\ka_1 + 1},\ldots, x_{\ka_1 + \ka_2}
    ) & \columnhom_{\ka_1 + \ka_2 + 1, \la}^N(
    v
    ) & \ldots & \columnhom_{N, \la}^N(x_{N}) 
\end{vmatrix}}%
{(-1)^{\frac{\ka_1(\ka_1 - 1)}{2} + \frac{\ka_2(\ka_2 - 1)}{2}}\,P_1 P_2}\\
&=
\frac{
\begin{vmatrix}
    A_2 & \columnhom_{\ka_1 + \ka_2 + 1, \la}^N(
    v
    ) & \ldots & \columnhom_{N, \la}^N(x_{N}) 
\end{vmatrix}}%
{(-1)^{\frac{\ka_1(\ka_1 - 1)}{2} + \frac{\ka_2(\ka_2 - 1)}{2}}\,P_1 P_2}.
\end{align*}
Step 3.
Apply Lemma~\ref{lem:reduction}
with $t=\ka_1 + \ka_2$,
$p=\ka_3$, $q=N - \ka_1 - \ka_2 - \ka_3$,
\[
u=(x_{\ka_1 + \ka_2 + 1},\ldots,x_{\ka_1 + \ka_2 + \ka_3}), \quad
v=(x_{\ka_1 + \ka_2 + \ka_3 + 1},\ldots,x_{N}).
\]
Then
\begin{align*}
D
&=
\frac{
\begin{vmatrix}
    A_2 & \columnhom_{N, \la}^N(x_{\ka_1 + \ka_2 + \ka_3}) & \ldots & \columnhom_{N - \ka_2 + 1, \la}^N(u
   ) & \columnhom_{\ka_1 + \ka_2 + 1, \la}^N(
    v
    ) & \ldots & \columnhom_{N, \la}^N(x_{N}) 
\end{vmatrix}}%
{
(-1)^{\frac{\ka_1(\ka_1 - 1)}{2} + \frac{\ka_2(\ka_2 - 1)}{2} + \frac{\ka_3(\ka_3 - 1)}{2}}\,
P_1 P_2 P_3
}\\
&=
\frac{
\begin{vmatrix}
    A_3 & \columnhom_{\ka_1 + \ka_2 + 1, \la}^N(
    v
    ) & \ldots & \columnhom_{N, \la}^N(x_{N}) 
\end{vmatrix}}%
{
(-1)^{\frac{\ka_1(\ka_1 - 1)}{2} + \frac{\ka_2(\ka_2 - 1)}{2} + \frac{\ka_3(\ka_3 - 1)}{2}}\,
P_1 P_2 P_3
}.
\end{align*}
Step $r$: Apply Lemma~\ref{lem:reduction}
with $t= \si(\ka)_{r-1} = \ka_1 + \ldots + \ka_{r-1}$,
$p=\ka_{r}$, $q=N - \si(\ka)_r$,
\[
u=(x_{\si(\ka)_{r-1} + 1},\ldots,x_{\si(\ka)_r}), \quad
v=(x_{\si(\ka)_r + 1},\ldots,x_{N}).
\] 
So,

\begin{align*}
D
&=
\frac{
\begin{vmatrix}
    A_{r-1} & \columnhom_{N, \la}^N(x_{\si(\ka)_r}) & \ldots & \columnhom_{N - \ka_r + 1, \la}^N(
    u
    ) & \columnhom_{\si(\ka)_r + 1, \la}^N(v) & \ldots & \columnhom_{N, \la}^N(x_{N}) 
\end{vmatrix}}%
{
(-1)^{\sum_{j=1}^r \frac{\ka_j(\ka_j - 1)}{2}}
\prod_{j=1}^{r} P_j 
}\\
&=
\frac{
\begin{vmatrix}
    A_{r} & \columnhom_{\si(\ka)_r + 1, \la}^N(v) & \ldots & \columnhom_{N, \la}^N(x_{N}) 
\end{vmatrix}}%
{
(-1)^{\sum_{j=1}^r \frac{\ka_j(\ka_j - 1)}{2}}
\prod_{j=1}^{r} P_j 
}.
\end{align*}
These steps are made for each $r$  from $1$ to $n$, resulting in
\[
D 
= \frac{\det A_n}{
(-1)^{\sum_{j=1}^n \frac{\ka_j(\ka_j - 1)}{2}}
\prod_{j=1}^{n} P_j }.
\]
After that, for each $j$ in  $\{1,\ldots, n\}$, we substitute $x_{\si(\ka)_{j-1} + 1},\ldots,x_{\si(\ka)_j}$ with $y_{j}$.
In the numerator, the $A_n$ transforms to $G_{\la}(y, \ka)$ in form~\eqref{eq:G_def_hom_rep}.
By~\eqref{eq:determinant_vandermonde_confluent_another_form}, in the denominator we obtain $\det G_\emptyset(y,\ka)$.
\end{proof}

\bigskip
\section{Third proof: using the classical bialternant formula}
\label{sec:proof_from_classic_bialternant}

In this section, we will apply \eqref{eq:hom_LitaDaSilva} in the following form:
\begin{equation}
\label{eq:da_silva_property_colum_h}
\begin{aligned}
&\columnhom_{q+1, \la}^N(t_1,\ldots,t_m,t_{m + 1})
-\columnhom_{q+1, \la}^N(t_1,\ldots,t_m,t_{m + 2})
\\
&\qquad\qquad
= (t_{m+1}-t_{m+2})
\columnhom_{q, \la}^N(t_1,\ldots,t_m,t_{m+1},t_{m+2}).
\end{aligned}
\end{equation}

Before the general reasoning,
we illustrate the idea with the following example.

\begin{ex} Let $n=2$, $\ka_1 = 3$, $\ka_2 = 2$,
and $x=(x_1,\ldots,x_5)$.
Then~\eqref{eq:schur_bialternant} reads as
\begin{align*}
\schur_{\la}(x)
&= \frac{\det
\begin{bmatrix}
\hom_{\la_1 + 4}(x_1) & \hom_{\la_1 + 4}(x_2) & \hom_{\la_1 + 4}(x_3) & \hom_{\la_1 + 4}(x_4) & \hom_{\la_1 + 4}(x_5)
\\
\hom_{\la_2 + 3}(x_1) & \hom_{\la_2 + 3}(x_2) & \hom_{\la_2 + 3}(x_3) & \hom_{\la_2 + 3}(x_4) & \hom_{\la_2 + 3}(x_5)
\\
\hom_{\la_3 + 2}(x_1) & \hom_{\la_3 + 2}(x_2) & \hom_{\la_3 + 2}(x_3) & \hom_{\la_3 + 2}(x_4) & \hom_{\la_3 + 2}(x_5)
\\
\hom_{\la_4 + 1}(x_1) & \hom_{\la_4 + 1}(x_2) & \hom_{\la_4 + 1}(x_3) & \hom_{\la_4 + 1}(x_4) & \hom_{\la_4 + 1}(x_5)
\\
\hom_{\la_5}(x_1) & \hom_{\la_5}(x_2) & \hom_{\la_5}(x_3) & \hom_{\la_5}(x_4) & \hom_{\la_5}(x_5)
\end{bmatrix}}{\det
\begin{bmatrix}
\hom_{4}(x_1) & \hom_{4}(x_2) & \hom_{4}(x_3) & \hom_{4}(x_4) & \hom_4(x_5)
\\
\hom_{3}(x_1) & \hom_{3}(x_2) & \hom_{3}(x_3) & \hom_{3}(x_4) & \hom_3(x_5)
\\
\hom_{2}(x_1) & \hom_{2}(x_2) & \hom_{2}(x_3) & \hom_{2}(x_4) & \hom_2(x_5)
\\
\hom_{1}(x_1) & \hom_{1}(x_2) & \hom_{1}(x_3) & \hom_{1}(x_4) & \hom_1(x_5)
\\
\hom_{0}(x_1) & \hom_{0}(x_2) & \hom_{0}(x_3) & \hom_{0}(x_4) & \hom_0(x_5)
\end{bmatrix}}.
\end{align*}
The determinant in the denominator is a particular case ($\la = \emptyset$) of the determinant in the numerator. We will simultaneously apply the same elementary operations to both of them and show the transformation of the numerator. 

\begin{itemize}
    \item Apply operation $\Add(2, 1, -1)$, apply \eqref{eq:da_silva_property_colum_h} and factorize $x_2 - x_1$.
    \item Apply operation $\Add(3, 1, -1)$, apply \eqref{eq:da_silva_property_colum_h} and factorize $x_3 - x_1$.
    \item Apply operation $\Add(3, 2, -1)$, apply \eqref{eq:da_silva_property_colum_h} and factorize $x_3 - x_2$.
\end{itemize}
We represent the numerator as the product of $(x_3 - x_1)(x_2 - x_1)(x_3 - x_2)$ by the following determinant:
\begin{align*}
\det
\begin{bmatrix}
\hom_{\la_1 + 4}(x_1) & \hom_{\la_1 + 3}(x_1,x_2) & \hom_{\la_1 + 2}(x_1, x_2, x_3) & \hom_{\la_1 + 4}(x_4) & \hom_{\la_1 + 4}(x_5)
\\
\hom_{\la_2 + 3}(x_1) & \hom_{\la_2 + 2}(x_1,x_2) & \hom_{\la_2 + 1}(x_1, x_2, x_3) & \hom_{\la_2 + 3}(x_4) & \hom_{\la_2 + 3}(x_5)
\\
\hom_{\la_3 + 2}(x_1) & \hom_{\la_3 + 1}(x_1, x_2) & \hom_{\la_3}(x_1, x_2, x_3) & \hom_{\la_3 + 2}(x_4) & \hom_{\la_3 + 2}(x_5)
\\
\hom_{\la_4 + 1}(x_1) & \hom_{\la_4}(x_1, x_2) & \hom_{\la_4 - 1}(x_1, x_2, x_3) & \hom_{\la_4 + 1}(x_4) & \hom_{\la_4 + 1}(x_5)
\\
\hom_{\la_5}(x_1) & \hom_{\la_5 - 1}(x_1, x_2) & \hom_{\la_5 - 2}(x_1, x_2, x_3) & \hom_{\la_5}(x_4) & \hom_{\la_5}(x_5)
\end{bmatrix}.
\end{align*}
The factor $(x_3 - x_1)(x_2 - x_1)(x_3 - x_2)$ also appears in the same way in the denominator, and this factor is canceled. Next, apply operation $\Add(5, 4, -1)$. Applying \eqref{eq:da_silva_property_colum_h} we  represent the numerator as the product of $x_5 - x_4$ by the following determinant:
\begin{align*}
\det
\begin{bmatrix}
\hom_{\la_1 + 4}(x_1) & \hom_{\la_1 + 3}(x_1,x_2) & \hom_{\la_1 + 2}(x_1, x_2, x_3) & \hom_{\la_1 + 4}(x_4) & \hom_{\la_1 + 3}(x_4,x_5)
\\
\hom_{\la_2 + 3}(x_1) & \hom_{\la_2 + 2}(x_1,x_2) & \hom_{\la_2 + 1}(x_1,x_2, x_3) & \hom_{\la_2 + 3}(x_4) & \hom_{\la_2 + 2}(x_4,x_5)
\\
\hom_{\la_3 + 2}(x_1) & \hom_{\la_3 + 1}(x_1, x_2) & \hom_{\la_3}(x_1, x_2, x_3) & \hom_{\la_3 + 2}(x_4) & \hom_{\la_3 + 1}(x_4, x_5)
\\
\hom_{\la_4 + 1}(x_1) & \hom_{\la_4}(x_1, x_2) & \hom_{\la_4 - 1}(x_1, x_2, x_3) & \hom_{\la_4 + 1}(x_4) & \hom_{\la_4}(x_4, x_5)
\\
\hom_{\la_5}(x_1) & \hom_{\la_5 - 1}(x_1, x_2) & \hom_{\la_5 - 2}(x_1, x_2 x_3) & \hom_{\la_5}(x_4) & \hom_{\la_5 - 1}(x_4, x_5)
\end{bmatrix}.
\end{align*}
The factor $(x_5 - x_4)$ also appears in the denominator.
Therefore, it is canceled.
We have obtained the following quotient:
\begin{align*}
\schur_\la(x) =  \frac{\det
\begin{bmatrix}
\hom_{\la_1 + 4}(x_1) & \hom_{\la_1 + 3}(x_1,x_2) & \hom_{\la_1 + 2}(x_1, x_2, x_3) & \hom_{\la_1 + 4}(x_4) & \hom_{\la_1 + 3}(x_4,x_5)
\\
\hom_{\la_2 + 3}(x_1) & \hom_{\la_2 + 2}(x_1,x_2) & \hom_{\la_2 + 1}(x_1,x_2, x_3) & \hom_{\la_2 + 3}(x_4) & \hom_{\la_2 + 2}(x_4,x_5)
\\
\hom_{\la_3 + 2}(x_1) & \hom_{\la_3 + 1}(x_1, x_2) & \hom_{\la_3}(x_1, x_2, x_3) & \hom_{\la_3 + 2}(x_4) & \hom_{\la_3 + 1}(x_4, x_5)
\\
\hom_{\la_4 + 1}(x_1) & \hom_{\la_4}(x_1, x_2) & \hom_{\la_4 - 1}(x_1, x_2, x_3) & \hom_{\la_4 + 1}(x_4) & \hom_{\la_4}(x_4, x_5)
\\
\hom_{\la_5}(x_1) & \hom_{\la_5 - 1}(x_1, x_2) & \hom_{\la_5 - 2}(x_1, x_2 x_3) & \hom_{\la_5}(x_4) & \hom_{\la_5 - 1}(x_4, x_5)
\end{bmatrix} }{\det
\begin{bmatrix}
\hom_{4}(x_1) & \hom_{3}(x_1,x_2) & \hom_{2}(x_1, x_2, x_3) & \hom_{4}(x_4) & \hom_{3}(x_4,x_5)
\\
\hom_{3}(x_1) & \hom_{2}(x_1,x_2) & \hom_{1}(x_1,x_2, x_3) & \hom_{3}(x_4) & \hom_{2}(x_4,x_5)
\\
\hom_{2}(x_1) & \hom_{1}(x_1, x_2) & \hom_{0}(x_1, x_2, x_3) & \hom_{2}(x_4) & \hom_{1}(x_4, x_5)
\\
\hom_{1}(x_1) & \hom_{0}(x_1, x_2) & 0 & \hom_{1}(x_4) & \hom_{0}(x_4, x_5)
\\
\hom_{0}(x_1) & 0 & 0 & \hom_{0}(x_4) & 0
\end{bmatrix} }.
\end{align*}
Substituting $x_1,x_2,x_3$ for $y_1$ and $x_4,x_5$ for $y_2$ in the previous formula and applying \eqref{eq:hom_one_variable_with_repetitions}, we conclude that 
$\schur_\la(y_1^{[3]}, y_2^{[2]})$
equals the following quotient:
\small{
\[
\frac{\det
\begin{bmatrix}
\binom{\la_1 + 4}{0}\hom_{\la_1 + 4}(y_1) & \binom{\la_1 + 4}{1}\hom_{\la_1 + 3}(y_1) & \binom{\la_1 + 4}{2}\hom_{\la_1 + 2}(y_1) & \binom{\la_1 + 4}{0}\hom_{\la_1 + 4}(y_2) & \binom{\la_1 + 4}{1}\hom_{\la_1 + 3}(y_2)
\\[1ex]
\binom{\la_2 + 3}{0}\hom_{\la_2 + 3}(y_1) 
&
\binom{\la_2 + 3}{1}\hom_{\la_2 + 2}(y_1) 
&
\binom{\la_2 + 3}{2}\hom_{\la_2 + 1}(y_1) & \binom{\la_2 + 3}{0}\hom_{\la_2 + 3}(y_2) & \binom{\la_2 + 3}{1}\hom_{\la_2 + 2}(y_2)
\\[1ex]
\binom{\la_3 + 2}{0}\hom_{\la_3 + 2}(y_1) & \binom{\la_3 + 2}{1}\hom_{\la_3 + 1}(y_1) & \binom{\la_3 + 2}{2}\hom_{\la_3}(y_1) & \binom{\la_3 + 2}{0}\hom_{\la_3 + 2}(y_2) & \binom{\la_3 + 2}{1}\hom_{\la_3 + 1}(y_2)
\\[1ex]
\binom{\la_4 + 1}{0}\hom_{\la_4 + 1}(y_1) & \binom{\la_4 + 1}{1}\hom_{\la_4}(y_1) & \binom{\la_4 + 1}{2}\hom_{\la_4 - 1}(y_1) & \binom{\la_4 + 1}{0}\hom_{\la_4 + 1}(y_2) & \binom{\la_4 + 1}{1}\hom_{\la_4}(y_2)
\\[1ex]
\binom{\la_5}{0}\hom_{\la_5}(y_1) & \binom{\la_5}{1}\hom_{\la_5 - 1}(y_1) & \binom{\la_5}{2}\hom_{\la_5 - 2}(y_1) & \binom{\la_5}{0}\hom_{\la_5}(y_2) & \binom{\la_5}{1}\hom_{\la_5 - 1}(y_2)
\end{bmatrix} }{\det
\begin{bmatrix}
\binom{4}{0}\hom_{4}(y_1) & \binom{4}{1}\hom_{3}(y_1) & \binom{4}{2}\hom_{2}(y_1) & \binom{4}{0}\hom_{4}(y_2) & \binom{4}{1}\hom_{3}(y_2)
\\[1ex]
\binom{3}{0}\hom_{3}(y_1) & \binom{3}{1}\hom_{2}(y_1) & \binom{3}{2}\hom_{1}(y_1) & \binom{3}{0}\hom_{3}(y_2) & \binom{3}{1}\hom_{2}(y_2)
\\[1ex]
\binom{2}{0}\hom_{2}(y_1) & \binom{2}{1}\hom_{1}(y_1) & \binom{2}{2}\hom_{0}(y_1) & \binom{2}{0}\hom_{2}(y_2) & \binom{2}{1}\hom_{1}(y_2)
\\[1ex]
\binom{1}{0}\hom_{1}(y_1) & \binom{1}{1}\hom_{0}(y_1) & 0 & \binom{1}{0}\hom_{1}(y_2) & \binom{1}{1}\hom_{0}(y_2)
\\[1ex]
\binom{0}{0}\hom_{0}(y_1) & 0 & 0 & \binom{0}{0}\hom_{0}(y_2) & 0
\end{bmatrix} }.
\]}

\noindent
In other words,
\[
\schur_\la\bigl(y_1^{[3]}, y_2^{[2]}\bigl)
=
\frac{G_\la(y, \ka)}{G_{\emptyset}(y, \ka)}.
\]
\hfill\qed
\end{ex}

\begin{proof}[Third proof of Theorem~\ref{thm:bialternant_formula_Schur_rep}.]
Let $x=(x_1, \ldots, x_N)$.
Applying notation~\eqref{eq:column_h} to~\eqref{eq:schur_bialternant}, we get
\begin{equation}
\label{eq:schur_bialternant_via_column_h}
\schur_\la(x)
=
\frac{\det
\begin{bmatrix}
\columnhom_{N,\la}^N(x_1) & \columnhom_{N,\la}^N(x_2) &
\ldots &
\columnhom_{N,\la}^N(x_N)
\end{bmatrix}}%
{\det
\begin{bmatrix}
\columnhom_{N,\emptyset}^N(x_1) & \columnhom_{N,\emptyset}^N(x_2) &
\ldots &
\columnhom_{N,\emptyset}^N(x_N)
\end{bmatrix}
}.
\end{equation}
Let $q\in\{1,\ldots,n\}$.
For brevity, put $\eta=\si(\ka)_{q-1}$.
We will apply some elementary operations to the columns with indices from
$\si(\ka)_{q-1}+1$ to $\si(\ka)_q$,
i.e., from $\eta+1$ to $\eta+\ka_q$.
Moreover, we will show only these columns.
With this emphasis,~\eqref{eq:schur_bialternant_via_column_h} can be rewritten as
\[
\schur_\la(x)
= \frac{\det
\begin{bmatrix}
\ \ldots &
\columnhom_{N,\la}^N(x_{\eta + 1}) &
\columnhom_{N,\la}^N(x_{\eta + 2}) &
\ldots &
\columnhom_{N,\la}^N(x_{\eta + \ka_q}) &
\ldots\ \mbox{}
\end{bmatrix}}%
{\det
\begin{bmatrix}
\ \ldots &
\columnhom_{N,\emptyset}^N(x_{\eta + 1}) &
\columnhom_{N,\emptyset}^N(x_{\eta + 2}) &
\ldots &
\columnhom_{N,\emptyset}^{N}(x_{\eta + \ka_q}) &
\ldots\ \mbox{}
\end{bmatrix}}.
\]
Step $1$.
For each $k$ in $\{2,\ldots,\ka_q\}$,
we apply operation
$\Add(\eta+k,\eta+1,-1)$,
i.e.,
we subtract column $\eta+1$
from column $\eta+k$
and transform column $\eta+k$
with help of~\eqref{eq:da_silva_property_colum_h}:
\[
\columnhom_{N,\la}^N(x_{\eta+k})
-\columnhom_{N,\la}^N(x_{\eta+1})
=(x_{\eta+k}-x_{\eta+1})
\columnhom_{N-1,\la}^N(x_{\eta+1},x_{\eta+k}).
\]
Then, we factorize the following product both in the numerator and denominator:
\[
\prod_{k=2}^{\ka_q}
\left(x_{\eta + k} - x_{\eta + 1}\right).
\]
This product is canceled,
and we obtain
\[
\schur_\la(x)
=
\frac{ \det
\begin{bmatrix} 
\ \ldots &
\columnhom_{N,\la}^N(x_{\eta + 1}) &
\columnhom_{N-1,\la}^N(x_{\eta + 1}, x_{\eta + 2}) &
\ldots &
\columnhom_{N-1,\la}^N(x_{\eta + 1}, x_{\eta + \ka_{q}}) &
\ldots\ \mbox{}
\end{bmatrix}}%
{\det
\begin{bmatrix}
\ \ldots &
\columnhom_{N,\emptyset}^N(x_{\eta + 1}) &
\columnhom_{N-1,\emptyset}^N(x_{\eta + 1}, x_{\eta + 2}) &
\ldots &
\columnhom_{N-1,\emptyset}^N(x_{\eta + 1}, x_{\eta + \ka_{q}}) &
\ldots\ \mbox{}
\end{bmatrix}}.
\]
Step $t$.
For each $k$ in $\{t + 1,\ldots,\ka_q\}$, we apply operation $\Add(\eta + k, \eta + t, -1)$. Using \eqref{eq:da_silva_property_colum_h}, we factorize the following product both in the numerator and denominator:
\[
\prod_{k=t + 1}^{\ka_q}\left(x_{\eta + k} - x_{\eta + t}\right).
\]
This product is canceled. Using notation $u = (x_{\eta + 1}, \ldots, x_{\eta + t})$, we obtain the following quotient of determinants:
\[
\frac{ \det
\begin{bmatrix}
\ldots &
\columnhom_{N,\la}^{N}(x_{\eta + 1}) & \ldots & \columnhom_{N - t + 1,\la}^{N}(u) & \columnhom_{N - t,\la}^{N}(u, x_{\eta + t +1}) & \ldots & \columnhom_{N- t, \la}^{N}(u, x_{\eta + \ka_{q}}) & \ldots
\end{bmatrix}}{\det 
\begin{bmatrix}
\ldots & \columnhom_{N,\emptyset}^{N}(x_{\eta + 1}) & \ldots & \columnhom_{N - t + 1,\emptyset}^{N}(u) & 
\columnhom_{N- t ,\emptyset}^{N}(u,x_{\eta + t + 1}) & \ldots & \columnhom_{N-t,\emptyset}^{N}(u,x_{\eta + \ka_{q}}) & \ldots
\end{bmatrix}}.
\]

These steps are made for each $t$ from $1$ to $\ka_q - 1$,
in the ascending order.
After that, we substitute $x_{\eta + 1},\ldots,x_{\eta + \ka_q}$  with $y_q$. Thereby, we get
\[
\frac{ \det
\begin{bmatrix}
\ldots &
\columnhom_{N,\la}^{N}(y_q) & \columnhom_{N - 1,\la}^{N}(y_q^{[2]}) & \ldots  & \columnhom_{N- \ka_q + 1, \la}^{N}(y_q^{[\ka_q]}) & \ldots
\end{bmatrix}}{ \det
\begin{bmatrix}
\ldots &
\columnhom_{N,\emptyset}^{N}(y_q) & \columnhom_{N - 1,\emptyset}^{N}(y_q^{[2]}) & \ldots  & \columnhom_{N- \ka_q + 1, \emptyset}^{N}(y_q^{[\ka_q]}) & \ldots
\end{bmatrix}}.
\]

We perform all the operations described above for each $q$ in $\{1, \ldots, n\}$, and we get the right-hand side of \eqref{eq:bialternant_formula_Schur_rep} with matrices $G_\la(y,\ka)$ and $G_{\emptyset}(y,\ka)$ in form~\eqref{eq:G_def_hom_rep}.
\end{proof}

\bigskip

\section{\texorpdfstring{Computation of the matrix $\boldsymbol{G_{\la}(y,\ka)}$}{Computation of the matrix G}}
\label{sec:computation_of_G}

The programs mentioned in this section are freely available at~\cite{GonzalezSerranoMaximenko2023schurgithub}. 
In this section,
we use the zero-based numbering.
The next simple code, written in the sintaxis of Python/SageMath~\cite{Sagemath},
constructs matrix $G_\la(y,\ka)$ from the initial data $\la$, $y$, $\ka$
by formula~\eqref{eq:G0_def}.
We suppose that the components of $y$ are of the same type $T$,
where $T$ is a field.
For example, $T$ can be $\bQ$,
or an approximate version of $\bC$,
or $\bQ(y_0,\ldots,y_{n-1})$,
i.e., the fraction field of multivariate polynomial ring in $y_0,\ldots,y_{n-1}$ with coefficients in $\bQ$.
At the beginning, we extend $\la$ by zeros to the length $N$.
We fill the resulting matrix by columns.
To avoid the conditional expression if/else, for each column we compute the maximum of the indices $j$ such that
$N+\la_j-j-r-1 \ge 0$.
We recall that \texttt{range(n)} in Python constructs the list $0,1,\ldots,n-1$.
We write \texttt{la} instead of \texttt{lambda} because \texttt{lambda} is a keyword in Python.

\begin{alg}[trivial algorithm to construct $G_\la(y,\ka)$]
\mbox{}
\label{alg:trivial_G}
\begin{lstlisting}
def gen_confluent_vandermonde_matrix(la, y, kappa):
    n = len(kappa)
    N = sum(kappa)
    la_ext = la + [0] * (N - n)  
    T = parent(y[0])
    G = matrix(T, N, N)
    p = 0
    for q in range(n):
        for r in range(kappa[q]):
            u = [N + la_ext[j] - j - r - 1 for j in range(N)]
            jlist = [j for j in range(N) if u[j] >= 0]
            jmax = max(jlist, default = -1)
            for j in range(jmax + 1):
                G[j, p] = binomial(u[j], r) * (y[q] ** u[j])
            p += 1
    return G
\end{lstlisting}
\end{alg}

\begin{rem}[time complexity for finite-precision fields]
Suppose that $T$ is a
finite-precision version of $\bR$ or $\bC$.
Then the powers $a^b$ can be computed as $\exp(b\log a)$
and the computation of the binomial coefficients can also be performed in $T$.
After these modifications,
Algorithm~\ref{alg:trivial_G}
has time complexity $O(N^2)$.
Notice that the determinant of $G_\la(y,\ka)$ can be computed by some of classical algorithms (a variant of PLU or QR)
involving $O(N^3)$ arithmetic operations in $T$,
and the determinant of $G_\emptyset(y,\ka)$ can be computed by~\eqref{eq:determinant_vandermonde_confluent} using $O(N^2)$ arithmetic operations.
So, this method to compute $\schur_\la(y^{[\ka]})$ has time complexity $O(N^3)$.
\end{rem}

\medskip\noindent
In the rest of this section,
we suppose that the powers $a^b$ are not computed as $\exp(b\log a)$.
We explain an algorithm to construct $G_\la(y,\ka)$
which is slightly more efficient than Algorithm~\ref{alg:trivial_G}.
It can be useful if $T$ is an exact field, such as $\bQ$ or $\bQ(y_0,\ldots,y_{n-1})$,
and $\la$ has large components.

Let $K$ be the maximum value of the list $\ka$ and
\[
L\eqdef \lfloor \log_2(\la_0+N)\rfloor,
\qquad\text{i.e.},\qquad
L =
\max
\bigl\{
k\in\bNz\colon\
2^k \le \la_0+N
\bigr\}.
\]
We use the following ideas.
\begin{enumerate}

\item Before filling \texttt{G},
we compute the binomial coefficients
$\binom{N+\la_j-j-1}{r}$
for $0\le r<K$.

\item For each $q$,
we compute the ``binary powers'' $y_q,y_q^2,y_q^4,y_q^8,\ldots,y_q^{2^L}$
by squaring.

\item To calculate $y_q^{N+\la_j-j-r-1}$,
we apply the binary exponentiation algorithm (also known as exponentiating by squaring)
using the precomputed array
$[y_q,y_q^2,y_q^4,y_q^8,\ldots,y_q^{2^L}]$.

\end{enumerate}

\noindent
Here are some auxiliary functions.

\begin{lstlisting}
def binomial_coefs(t, m):
    # compute [C[0], ..., C[m - 1]] where C[k] = binomial(t, k)
    C = vector(ZZ, m)
    C[0] = 1  
    for k in range(1, m):
        C[k] = C[k - 1] * (t - k + 1) / k
    return C
\end{lstlisting}

\begin{lstlisting}
def integer_log2(x):
    # compute floor(log2(x)), i.e.,
    # the maximum of k such that 2 ** k <= x
    y = 1
    k = 0
    while y <= x:
        y = 2 * y
        k += 1
    return k - 1
\end{lstlisting}

\begin{lstlisting}
def binary_powers(a, L):
    # compute [a, a ** 2, a ** 4, a ** 8, ..., a ** (2 ** L)]
    F = parent(a)
    b = vector(F, L + 1)
    p = 1
    c = a
    for j in range(L + 1):
        b[j] = c
        c = c * c
    return b
\end{lstlisting}

\begin{lstlisting}
def binary_expon(B, p):
    # compute a ** p using
    # B = [a, a ** 2, a ** 4, a ** 8, ..., a ** (2 ** L)];
    # we suppose that p < 2 ** (L + 1)
    F = parent(B[0])
    c = F.one()
    q = p
    j = 0
    while q > 0:
        if mod(q, 2) != 0:
            c = c * B[j]
        q = q // 2
        j = j + 1
    return c
\end{lstlisting}

\begin{alg}[an efficient algorithm to construct $G_\la(y,\ka)$]
\label{alg:G_efficient}
\mbox{}
\begin{lstlisting}
def gen_confluent_vandermonde_matrix_eff(la, y, kappa):
    n = len(kappa)
    N = sum(kappa)
    K = max(kappa)
    la_ext = la + [0] * (N - n) 
    L = integer_log2(la_ext[0] + N - 1)
    T = parent(y[0])
    C = matrix(ZZ, N, K)
    for j in range(N):
        C[j, :] = binomial_coefs(N + la_ext[j] - j - 1, K)
    G = matrix(T, N, N)
    p = 0
    for q in range(n):
        B = binary_powers(y[q], L)
        for r in range(kappa[q]):
            u = [N + la_ext[j] - j - r - 1 for j in range(N)]
            jlist = [j for j in range(N) if u[j] >= 0]
            jmax = max(jlist, default = -1)
            for j in range(jmax + 1):
                G[j, p] = C[j, r] * binary_expon(B, u[j])
            p += 1            
    return G
\end{lstlisting}
\end{alg}

\begin{rem}[an upper bound for the number of arithmetic operations over $T$ in Algorithm~\ref{alg:G_efficient}]
\label{rem:G_efficient_computational_complexity}
Let $A(\la,\ka)$ be the number of arithmetic operations over $T$ or $\bZ$
performed in Algorithm~\ref{alg:G_efficient}.
We are going to estimate $A(\la,\ka)$ from above.
Here are the principal computational costs:
\begin{itemize}

\item $O(N K)$
arithmetic operations
to fill \texttt{C},
i.e., the matrix of the binomial coefficients;

\item $O(n L)$
arithmetic operations
to compute \texttt{B}
for all values of $q$;

\item $O(N^2 L)$
arithmetic operations
to compute all components of \texttt{G},
using the precomputed matrices
\texttt{C} and \texttt{B}.

\end{itemize}
So,
\[
A(\la, \ka)
= O(NK) + O(nL) + O(N^2 L).
\]
Since $n\le N$ and $K\le N$, we obtain the following upper estimate:
\begin{equation}
\label{eq:G_efficient_number_of_operations}
A(\la, \ka)
= O(N^2 L)
= O(N^2 \log_2(\la_0 + N)).
\end{equation}
\end{rem}

\begin{rem}[an upper bound for the number of arithmetic operations to compute
$\schur_\la(y^{[\ka]})$]
\label{rem:upper_bound_number_operations_for_schur}
Using the binary exponentiation,
we apply $O(n^2 \log(K))$
arithmetic operations in $T$
to compute
$G_{\emptyset}(y, \ka)$ via~\eqref{eq:determinant_vandermonde_confluent}.
Furthermore, we assume that the determinant of $G_\la(y,\ka)$ is computed by some of the classical algorithms
(a variant of PLU) involving $O(N^3)$ arithmetic operations in $T$.
Then, the total number of arithmetic operations is
\[
O(N^2 L) + O(N^3).
\]
Since $L\le \log_2(\la_0+1) + \log(N) + 1$,
we have
$O(N^2 L)\le O(N^2 \log_2(\la_0+1)) + O(N^2 \log_2(N))$,
and the last term is absorbed by $O(N^3)$.

We get the following final estimate for the number of operations in $T$ or $\bZ$ needed to compute
$\schur_\la(y^{[\ka]})$
via Algorithm~\ref{alg:G_efficient}:
\begin{equation}
\label{eq:bialternant_formula_number_of_operations}
O(N^3)+O(N^2 \log_2 (\la_0+1)).
\end{equation}
\end{rem}

\begin{rem}
The analysis in Remarks~\ref{rem:G_efficient_computational_complexity} and~\ref{rem:upper_bound_number_operations_for_schur} is naive in the following sense.
If, for example, $T=\bQ$ or $T=\bQ(y_0,\ldots,y_{n-1})$ and $\la_0+N$ is large,
then Algorithm~\ref{alg:G_efficient}
involves arithmetic operations
over large integers or polynomials with many terms.
In such cases, the arithmetic operations in $T$ do not have a constant complexity,
and counting the arithmetic operations in $T$ is not sufficient to estimate the time complexity of the algorithm.
For the last purpose,
a much deeper and complicated analysis is required.
\end{rem}

\begin{rem}[tests of the main formulas]
We have tested~\eqref{eq:bialternant_formula_Schur_rep}
in SageMath using symbolic computations with multivariate rational functions over the field $\bQ$ of rational numbers.
The left-hand side of~\eqref{eq:bialternant_formula_Schur_rep} has been computed with the Jacobi--Trudi formula.
In this manner, we have verified~\eqref{eq:bialternant_formula_Schur_rep}
for all $n$ in $\bN$ with $n\le 8$,
all $\ka$ in $\bN^n$ with $N\eqdef|\ka|\le 8$,
and all integer partitions $\la$ of length $\le N$
with $|\la|\le 8$.
For all these tests, Algorithms~\ref{alg:trivial_G} and \ref{alg:G_efficient} yield the same matrix.
We have also tested~\eqref{eq:determinant_vandermonde_confluent},
\eqref{eq:determinant_vandermonde_confluent_another_form},
\eqref{eq:alternating},
\eqref{eq:product_G_equal_H_M},
and~\eqref{eq:partial_reduction_JT}.
\end{rem}

\medskip

\section{Connection with plethysm}
\label{sec:connection_with_plethysm}

In this section,
we consider the case if
$\ka_1=\ldots=\ka_n=k$,
and show that
\[
\schur_\la\bigl(\,
\underbrace{y_1,\ldots,y_1}_{k\ \text{times}},\,\ldots,\,
\underbrace{y_n,\ldots,y_n}_{k\ \text{times}}\,
\bigr)
\]
can be written in terms of the plethysm of symmetric functions (see, e.g., \cite[Chapter~1]{Macdonald1995}).
Moreover, instead of $\schur_\la$,
we deal with an arbitrary symmetric function.

We denote by $\circ$ the plethysm operation on symmetric functions.
Given $s$ in $\bN$,
let $\powersum_s$ be the power sum:
\[
\powersum_s(x_1,x_2,x_3,\ldots)
\eqdef x_1^s + x_2^s + x_3^s + \ldots.
\]
If $\la=(\la_1,\ldots,\la_q)$ is a partition with $\la_1\ge\ldots\ge\la_q \ge 1$,
then, by definition,
\[
\powersum_\la
\eqdef
\powersum_{\la_1}\cdots\powersum_{\la_q}.
\]
Moreover, by definition,
$\powersum_\emptyset \eqdef 1$.

\begin{prop}
\label{prop:rep_vs_plethysm}
Let $f$ be a symmetric function,
$n,k\in\{1,2,\ldots\}$,
$y=(y_1,\ldots,y_n)$,
\[
\ka
=k^{[n]}
=\bigl(\,\underbrace{k,\ldots,k}_{n\ \text{times}}\,\bigr),
\qquad
y^{[\ka]}
= \bigl(\,
\underbrace{y_1,\ldots,y_1}_{k\ \text{times}},\ 
\ldots,\ 
\underbrace{y_n,\ldots,y_n}_{k\ \text{times}}\,
\bigr).
\]
Then
\begin{equation}
\label{eq:f_rep_eq_f_plethysm_kp1}
f(y^{[\ka]})
=(f \circ (k\powersum_1))(y).
\end{equation}
\end{prop}

\begin{proof}
First, we notice that if $s\in\{0,1,2,\ldots\}$, then
\[
\powersum_s(y^{[\ka]})
=\sum_{j=1}^n k y_j^s
=k \powersum_s(y).
\]
Let $f=\powersum_{\la_1,\ldots,\la_q}$,
where $\la_1\ge\ldots\ge\la_q$.
Then
\begin{align*}
f(y^{[\ka]})
&=
\powersum_{
\la_1,\ldots,
\la_q}(y^{[\ka]})
=
\prod_{r=1}^q
\powersum_{\la_r}(y^{[\ka]})
=
\prod_{r=1}^q
\bigl(k \powersum_{\la_r}(y)\bigr)
=
k^q
\prod_{r=1}^q
\powersum_{\la_r}(y)
=k^q f(y).
\end{align*}
On the other hand,
\begin{align*}
(f\circ (k \powersum_1))(y)
&=
\bigl(\powersum_{\la_1,\ldots,\la_q}\circ(k\powersum_1)\bigr)(y)
=
\prod_{r=1}^q
\bigl(\powersum_{\la_r}\circ (k\powersum_1)\bigr)(y)
\\
&=
\prod_{r=1}^q
\bigl(k \powersum_{\la_r}(y)\bigr)
=k^q
\prod_{r=1}^q
\powersum_{\la_r}(y)
=k^q f(y).
\end{align*}
We have proven~\eqref{eq:f_rep_eq_f_plethysm_kp1} for $f=\powersum_{\la_1,\ldots,\la_q}$.
Obviously, \eqref{eq:f_rep_eq_f_plethysm_kp1} is also true if $f=\powersum_\emptyset$.
It is well known~\cite[(2.12)]{Macdonald1995}
that every symmetric function is a linear combination of power sums,
and the plethysm operation is linear with respect to the first argument~\cite[(8.3)]{Macdonald1995}.
Thereby, \eqref{eq:f_rep_eq_f_plethysm_kp1}
is extended to the general case.
\end{proof}

\begin{ex}
Let $f=\schur_{(1,1)}$
and $\ka=(2,2)$.
First, we consider a Schur polynomial in four variables:
\[
\schur_{(1,1)}(x_1,x_2,x_3,x_4) 
= x_1x_2 + x_1x_3 + x_1x_4 + x_2x_3 + x_2x_4 + x_3x_4.
\]
Then, we use only two variables, but we repeat each of them twice:
\begin{equation}
\label{eq:schur11_y1y1_y2y2}
\schur_{(1,1)}(y_1^{[2]}, y_2^{[2]}) 
= \schur_{(1,1)}(y_1,y_1,y_2,y_2) 
= y_1^2 + 4y_1y_2 + y_2^2.
\end{equation}
It is easy to see that $\schur_{(1,1)}$
is the following
linear combination of power sums:
\[
\schur_{(1,1)}
=\frac{1}{2}
\powersum_1
\powersum_1
-\frac{1}{2}
\powersum_2.
\]
By the definition of plethysm for power sums,
\[
\schur_{(1,1)}
\circ\bigl(2\powersum_1\bigr)
=
\left(
\frac{1}{2}
\powersum_1
\powersum_1
-\frac{1}{2}
\powersum_2
\right)
\circ(2\powersum_1)
=
\frac{1}{2}
\left(2\powersum_1\right)
\left(2\powersum_1\right)
-\frac{1}{2}\cdot 2\powersum_2
=
2\powersum_1\powersum_1
-\powersum_2.
\]
Restricting this symmetric function to two variables,
$y_1$ and $y_2$,
we obtain
\begin{equation}
\label{eq:s11_2p1_y1y2}
\left(\schur_{(1,1)}
\circ\bigl(2\powersum_1\bigr)\right)(y_1,y_2)
=
2\bigl(y_1+y_2\bigr)^2
-\left(y_1^2+y_2^2\right)
=
y_1^2 + 4y_1 y_2 + y_2^2.
\end{equation}
Notice that~\eqref{eq:schur11_y1y1_y2y2}
coincides with~\eqref{eq:s11_2p1_y1y2}.
\hfill$\qedsymbol$
\end{ex}

\section*{Acknowledgements}

We are grateful to Prof. Per Alexandersson
(\myurl{https://orcid.org/0000-0003-2176-0554})
for the idea of Section~\ref{sec:connection_with_plethysm}.

\section*{Declaration of competing interest}

There is no competing interest.

\label{endlabel}
\end{document}